\documentclass[11pt,letterpaper]{article}
\usepackage[latin1]{inputenc}
\usepackage{amsmath}
\usepackage{amsfonts}
\usepackage{amssymb}
\usepackage{amsthm}
\usepackage{graphicx}   
\usepackage{epstopdf}
\usepackage{subfigure}
\usepackage{pifont}

\title{$F$-Signature of Affine Toric Varieties}			
\author{Michael Von Korff}		
\date{October 2011}					

\DeclareMathOperator{\Lattice}{Lattice}
\DeclareMathOperator{\im}{im}
\DeclareMathOperator{\Vol}{Volume}

\DeclareMathOperator{\Hom}{Hom}
\DeclareMathOperator{\Frac}{Frac}

\DeclareMathOperator{\Spec}{Spec}
\DeclareMathOperator{\Rad}{Rad}
\DeclareMathOperator{\Newt}{Newt}
\DeclareMathOperator{\dual}{\vee}
\DeclareMathOperator{\PDa}{P_{\sigma}^{D,\mathfrak{a}}}
\newcommand{\gr}{gr}
\newcommand{\divf}{\mbox{div }}
\newcommand{\into}{\hookrightarrow}
\newcommand{\onto}{\twoheadrightarrow}

\newcommand{\suchthat}{\,|\:}
\newtheorem*{ThmAffineToricFSignature}{Theorem \ref{AffineToricFSignature}}
\newtheorem*{ThmFSignatureOfTriples}{Theorem \ref{FSignatureOfTriples}}
\newtheorem*{ThmFSignatureOfPairs}{Theorem \ref{FSignatureOfPairs}}
\newtheorem*{CorGorensteinTriples}{Corollary \ref{GorensteinTriples}}

\newtheorem{thm}{Theorem}[section]
\newtheorem{cor}[thm]{Corollary}
\newtheorem{defn}[thm]{Definition}
\newtheorem{lem}[thm]{Lemma}
\theoremstyle{remark}
\newtheorem{rem}[thm]{Remark}

\newtheorem{ex}[thm]{Example}

\begin{document}

\maketitle

\section{Introduction}

The $F$-signature is a numerical invariant of a local (or graded) ring $R$ containing a field of positive characteristic. It arose naturally in \cite{SV97} in the study of the ring of differential operators in prime characteristic but was first singled out as an object of independent interest in \cite{HL04}. It has been shown to be intimately connected to various measures of the singularity of $R$. For example, the $F$-signature is positive if and only if the ring is strongly $F$-regular, and it is equal to 1 if and only if the ring is regular \cite{HL04}. (In general, it is a real number between 0 and 1.)

In this paper, we will compute the $F$-signature when $R$ is the coordinate ring of an affine toric variety or, equivalently, a normal monomial ring. (The $F$-signature of a non-normal ring is zero in any case.) In particular:

\begin{ThmAffineToricFSignature}(cf. \cite{WY04}, Theorem 5.1)
Let $R=k[S]$ be the coordinate ring of an affine toric variety $X$ without torus factors, with the conventions of Remark \ref{ToricAssumptions} below. In particular, $S = M\cap \sigma^{\dual}$, where $M$ is a lattice, $\sigma\subset M_{\mathbb{R}}=M\otimes_{\mathbb{Z}} \mathbb{R}$ is a full-dimensional strongly convex rational polyhedral cone, and $\sigma^{\dual}$ is the dual cone to $\sigma$. Let $\vec{v}_1,\ldots,\vec{v}_r\in\mathbb{Z}^n$ be primitive generators for $\sigma$. Let $P_{\sigma}\subset \sigma$ be the polytope $\{\vec{w}\in M_{\mathbb{R}}\suchthat \forall i, 0\leq \vec{w}\cdot \vec{v}_i< 1\}$. Then $s(R) = \Vol(P_{\sigma})$.

More generally, suppose $X = X' \times T$, where $X'$ is a toric variety without torus factors and $T$ is an algebraic torus. Let $N'_{\mathbb{R}}\subset N_{\mathbb{R}}$ be the vector subspace spanned by $\sigma$, and let $\sigma'$ be $\sigma$ viewed as a cone in $N'_{\mathbb{R}}$. Then $s(k[X]) = s(k[X'])=\Vol(P_{\sigma'})$.
\end{ThmAffineToricFSignature}

(We will review the notation of cones and toric varieties in the next section.) Our formula is equivalent to the one given in \cite{WY04} in the case where $X$ has no torus factors; when $X$ does have torus factors, our formula corrects the one given in \cite{WY04}, which does not hold in that case.

Our method of proof differs from the proof given in \cite{WY04}. It is similar to that used by Watanabe in \cite{Wat99} to compute toric Hilbert-Kunz multiplicities. A different formula for the $F$-signature of a normal monomial ring has also been computed by Singh in \cite{Sin05}. The methods used in this paper allow us to give an easy proof of Singh's result (Theorem \ref{SinghTheorem}).

The notion of $F$-signature of a pair $(R, D)$ or a triple $(R, D, \mathfrak{a}^t)$ has been defined in \cite{BST11}. We compute these in the toric case:

\begin{ThmFSignatureOfPairs}
Let $R$ be the coordinate ring of an affine toric variety, with conventions as in Remark \ref{ToricAssumptions}. Let $D$ be a torus-invariant divisor, with associated polytope $P_{\sigma}^D$ as in Definition \ref{PDDef}. Then $s(R) = \Vol(P_{\sigma}^D)$.
\end{ThmFSignatureOfPairs}

\begin{ThmFSignatureOfTriples}
Let $R$ be the coordinate ring of an affine toric variety, with conventions as in Remark \ref{ToricAssumptions}. Let $D$ be a torus-invariant divisor as in Definition \ref{PDDef}. Let $\mathfrak{a}\subset R$ be a monomial ideal, with associated polytope $\PDa$ as in Definition \ref{PaDef}. Then $s(R, D, \mathfrak{a}^t) = \Vol(\PDa)$.
\end{ThmFSignatureOfTriples}

\begin{CorGorensteinTriples}
Let $R$ be the coordinate ring of an affine toric variety, $D$ a divisor on $\Spec R$, and $\mathfrak{a}$ a monomial ideal, presented as in Theorem \ref{FSignatureOfTriples}. Suppose that the pair $(R, D)$ is $\mathbb{Q}$-Gorenstein. Then $s(R, D, \mathfrak{a}^t) = \Vol(P_{\sigma}^D\cap t\cdot \Newt(\mathfrak{a}))$.
\end{CorGorensteinTriples}

I would like to thank my advisor, Karen Smith, for providing guidance during my work on this paper; Kevin Tucker for his insight into $F$-signature of pairs and triples, and for providing a proof of Lemma \ref{TriplesInterp}; and Julian Rosen for several useful discussions which led to a first proof of Lemma \ref{PCCharacterization}. (The proof given here is inspired by \cite{Wat99}.)

This work was partially supported by NSF grant DMS-0502170.

\section{Preliminaries}

\subsection{F-signature}

We recall the definition of $F$-signature. Let $R$ be a ring containing a field $k$ of characteristic $p>0$. Let $F^e_*R$ be the $R$-module whose underlying abelian group is $R$ and whose $R$-module structure is given by Frobenius: for $r\in R, s\in F^e_*R$, $r\cdot s = r^{p^e}s$. If $R$ is reduced, it's easy to see that $F^e_*R$ is isomorphic to $R^{1/p^e}$, the $R$-module of $p^e$th roots of elements of $R$. (This also gives $F^e_*R$ a natural ring structure.) Recall that $R$ is said to be \emph{$F$-finite} if $F_*R$ is a finitely generated $R$-module. (For example, every finitely generated algebra over a perfect field is $F$-finite.)

\begin{rem}\label{RingAssumptions}(Conventions.)
In what follows, all rings are assumed to be Noetherian $F$-finite domains containing a field $k$ of prime characteristic $p>0$. We assume that $k$ is perfect unless stated otherwise (though by Remark \ref{ImperfectResidueField}, this assumption is mostly without loss of generality). Moreover, all rings are either local with residue field $k$; graded over $\mathbb{Z}^n$ for some $n$ with each graded piece isomorphic to $k$; or graded over $\mathbb{N}$ with zeroeth graded piece is equal to $k$.
\end{rem}

\begin{defn}
Let $R$ be a Noetherian local ring or an $\mathbb{N}$-graded ring with zeroeth graded piece a equal to a field. Let $M$ be a finitely generated $R$-module (which is assumed to be graded if $R$ is graded), and consider a decomposition of $M$ as a direct sum of indecomposable $R$-modules. The \emph{free rank} of $M$ as an $R$-module is the number of copies of $R$ in this direct sum decomposition.
\end{defn}

\begin{rem}
In general (if $R$ is not local or $\mathbb{N}$-graded over a field), the free rank of a module is not well-defined. In the local or $\mathbb{N}$-graded setting, however, free rank is uniquely determined (see, e.g., Remark \ref{IeDef} and Lemma \ref{GradedToLocal}).
\end{rem}

\begin{defn}\label{FSignatureDefinition}\cite{HL04}
Let $R$ be a ring (either local or graded, as described above) of dimension $d$. Let $\alpha = \log_p [k^p:k]<\infty$. For each $e\in\mathbb{N}$, let $a_e$ be the free rank of $R^{1/p^e}$ as an $R$-module, so that $F^e_*R = R^{\oplus a_e}\oplus M_e$, with no copies of $R$ in the direct sum decomposition of $M_e$. We define the \emph{$F$-signature} of $R$ to be the limit $$s(R) = \lim_{e\to\infty} \frac{a_e}{p^{e(d+\alpha)}}.$$
\end{defn}

\begin{rem}
Tucker \cite{Tuc10} showed that the limit given in Definition \ref{FSignatureDefinition} exists when $R$ is a local ring. (It follows from Lemma \ref{GradedToLocal}.\ref{GradedEqualsLocal} that $F$-signature is well-defined when $R$ is $\mathbb{N}$-graded as in Remark \ref{RingAssumptions}.)
\end{rem}

For the sake of simplicity, we will confine ourselves to the case of perfect $k$ at first, so that $s(R) = \lim_{e\to\infty} \frac{a_e}{p^{ed}}$. After we have proved our main results in the perfect case, it will not be difficult to extend them to the case $[k^p:k]<\infty$.

\begin{rem}\label{IeDef}
The $F$-signature of a local ring $(R, m, k)$ may also be characterized as follows (see \cite{Tuc10}): define $I_e\subset R$ to be the ideal $\{r\in R\suchthat \forall \phi\in \Hom_R(R^{1/p^e}, R), \phi(r^{1/p^e})\in m\}$. In other words, $I_e$ is the ideal of elements of $R$ whose $p^e$th roots do not generate a free summand of $R^{1/p^e}$. Then $a_e = l(F_*^e(R/I_e))$, so $s(R) = \lim_{e\to\infty} \frac{l(F_*^e(R/I_e))}{p^{e(d+\alpha)}}$.

Since $l(F^e_*M) = [k^p:k]^e l(M)$, we arrive at the following definition of $F$-signature, which does not depend on $\alpha$: $$s(R) = \lim_{e\to\infty} \frac{l(R/I_e)}{p^{ed}}.$$
\end{rem}

\subsection{Affine toric varieties}\label{ToricVarietiesSection}

Here, we present enough background on toric varieties to prove Theorem \ref{AffineToricFSignature}. Almost all notation is standard as in Fulton's book \cite{Ful93}, which the reader may consult for further details.

A toric variety $X$ may be defined as a normal variety which contains an algebraic torus $T=\Spec k[x_1,x_1^{-1},\ldots, x_n, x_n^{-1}]$ as an open dense subset, so that the action of $T$ on itself extends to an action of $T$ on $X$. Toric varieties can be presented in terms of simple combinatorial data, making algebro-geometric computations easier on toric varieties.

Let $N$ be a free abelian group of rank $n$. Let $M = N^*=\Hom_{\mathbb{Z}}(N, \mathbb{Z})$ the dual group to $N$. Consider $M$ as a lattice, called the \emph{character lattice}, in the $\mathbb{R}$-vector space $M_{\mathbb{R}} := M\otimes_{\mathbb{Z}}\mathbb{R}$. Let $k[M]$ be the semigroup ring on $M$, so that up to non-canonical isomorphism, $k[M]\simeq k[x_1^{\pm 1},\ldots, x_n^{\pm 1}]$ is the coordinate ring of an ``algebraic torus." Elements of the semigroup $M$ are called \emph{characters} but may also be thought of as exponents; the inclusion of abelian groups $\chi: M^+\into (k[M])^{\times}$ is called the \emph{exponential map} and is written $m\mapsto \chi^m$. Elements $\chi^m\in k[M]$ are called \emph{monomials}. In this paper, a \emph{monomial ring} $R$ is a $k$-subalgebra of $k[M]$, finitely generated by monomials: $R = k[S]$, where $\chi^S$ is the set of monomials in $R$. Of course, the set of monomials in $R$ forms a semigroup under multiplication which is naturally isomorphic to $S$. We denote by $L=\Lattice(S)$ the (additive) subgroup of $M$ generated by $S$, which is isomorphic under the exponential map to the (multiplicative) group of monomials in $\Frac(R)$.

In what follows, let $\sigma\subset N_{\mathbb{R}}$ be a \emph{strongly convex rational polyhedral cone}. By \emph{rational polyhedral cone} we mean that $\sigma$ is the cone of vectors $\{\sum_i a_i\vec{v}_i\suchthat 0\leq a_i\in\mathbb{R}\}$, where $\vec{v}_i\in M$ are a collection of finitely many \emph{generators} for $\sigma$. Moreover, we require that $\sigma$ be \emph{strongly convex}: that is, if $0\neq\vec{v}\in \sigma$ then $-\vec{v}\notin \sigma$.

A minimal set of generators of a cone is uniquely determined up to rescaling. (For each $i$, $\mathbb{R}_{\geq 0}\cdot\vec{v}_i$ is a ray which forms one edge of the cone $\sigma\subset N_{\mathbb{R}}$.) It is often useful to take the vectors $\vec{v}_i$ to be \emph{primitive generators}: that is, we replace each $\vec{v}_i$ with the shortest vector in $N$ that lies on the same ray. The primitive generators of $\sigma$ are themselves uniquely determined.

A \emph{face} of $\sigma$ is $F = \sigma\cap H$, where $H\subset\mathbb{R}^n$ is a hyperplane that only intersects $\sigma$ on its boundary $\partial \sigma$. (Equivalently, $H = \vec{w}^{\perp}$, where $\vec{w}\cdot \vec{v}\geq 0$ for all $\vec{v}\in \sigma$. Such $H$ is called a \emph{supporting hyperplane}.) A codimension-one face is called a \emph{facet}. As is (hopefully) intuitively clear, one can show that the union of the facets of $\sigma$ is equal to the boundary of the cone, $\partial \sigma$. Every face of $\sigma$ is itself a strongly convex rational polyhedral cone, whose generators are a subset of the generators of $\sigma$.

A strongly convex rational polyhedral cone $\sigma$ in the vector space $V$ has a \emph{dual cone}, $\sigma^{\dual} = \{\vec{u}\in V^*\suchthat \vec{u}\cdot \vec{v}\geq 0 \forall \vec{v}\in \sigma\}$. It's a basic fact of convex geometry that $\sigma^{\dual}$ is also a rational polyhedral cone, and that $\sigma = (\sigma^{\dual})^{\dual}$. 



Now we define affine toric varieties in the language of cones. Every affine toric variety may be presented in the following form:

\begin{defn}
Let $N$ be an $n$-dimensional lattice, $N\subset N_{\mathbb{R}} = N\otimes_{\mathbb{Z}}\mathbb{R}$. Let $M = N^*$, and $M_{\mathbb{R}} = M\otimes\mathbb{R}$. Let $\sigma\subset N_{\mathbb{R}}$ be a strongly convex rational polyhedral cone, and $S = \sigma^{\dual}\cap M$, where $\sigma^{\dual}\subset M_{\mathbb{R}}$ is the dual cone to $\sigma$. Let $R = k[S]$. The affine toric variety corresponding to $\sigma$ is defined to be $X = \Spec R$.
\end{defn}

\begin{rem}
Only normal monomial rings arise as the coordinate rings of toric varieties. Since strongly $F$-regular rings are normal, there will be no loss of generality in restricting our $F$-signature computations to only those monomial rings arising from toric varieties. (If a monomial ring does not arise in this fashion, it is not normal, hence not strongly $F$-regular, so we already know that its $F$-signature is zero.)
\end{rem}

It will be convenient, during our $F$-signature computations, to temporarily assume that the cone $\sigma$ defining our toric variety $X_\sigma$ is full-dimensional. Equivalently, we assume that our toric variety contains no \emph{torus factors}, i.e, is not the product of two lower-dimensional toric varieties, one of which is a torus. The following (easily checked) facts about products of cones will allow us to reduce to the case of a toric variety with no torus factors:

\begin{lem}\label{ConeProduct}(\cite{CLS11}, Proposition 3.3.9)
Let $X = \Spec R$ be the affine toric variety corresponding to the cone $\sigma$, so that $R = k[\sigma^{\dual}\cap M]$. Let $N'_{\mathbb{R}}\subset N_{\mathbb{R}}$ be the vector subspace spanned by $\sigma$. Let $N' = N'_{\mathbb{R}} \cap N$. Let $\sigma'$ be $\sigma$, viewed as a full-dimensional cone in $N'_{\mathbb{R}}$. Let $N''=N/N'$. Then $X \simeq X'\times T_{N''}$, where $X'$ is the affine toric variety (with no torus factors) corresponding to $\sigma'$ and $T_{N''} = \Spec k[M'']$ is an algebraic torus.
\end{lem}


On a similar note, we see that full-dimensionality is a dual property to strong convexity:

\begin{lem}\label{FullDimensionalToStronglyConvex} (\cite{Ful93}, \S 1.2)
A (rational polyhedral) cone is full-dimensional if and only if its dual cone $\sigma^{\dual}$ is strongly convex. (Or, equivalently, if and only if $\vec{0}$ is the only unit in $\sigma^{\dual}\cap M$.)
\end{lem}

The following fact will also be useful later. It says that the group $\Lattice(S)$ generated by the semigroup $S$ is equal to the character lattice $M$.

\begin{lem}\label{ConeLattice}
Let $N$ be a rank-$n$ lattice, $M = N^*$, $\sigma\subset N_{\mathbb{R}}$ a strongly convex rational polyhedral cone, and $S = \sigma^{\dual}\cap M$ (so that $\Spec k[S]$ is the affine toric variety corresponding to $\sigma$). Then $\Lattice(S) = M$. More generally, if $L'\subset M_{\mathbb{R}}$ is any $n$-dimensional lattice, $\sigma^{\dual}\subset M_{\mathbb{R}}$ any $n$-dimensional cone, and $S = \sigma^{\dual}\cap L'$, then $\Lattice(S) = L'$, that is, $L'$ is the group generated by the semigroup $S$.
\end{lem}

Next, we define a polytope:

\begin{defn}
A \emph{polytope} in $\mathbb{R}^n$ is the convex hull of a finite set of points, which we will call extremal points. Equivalently, a polytope is a bounded set given as the intersection of finitely many closed half-spaces $H = \{\vec{v}\suchthat \vec{v}\cdot\vec{u}\geq 0\}$ (reference), or a bounded set defined by finitely many linear inequalities.
\end{defn}

\begin{rem}
We will abuse notation by allowing polytopes to be intersections of half-spaces which are either open ($H = \{\vec{v}\suchthat \vec{v}\cdot\vec{u}> 0\}$) or closed. (We will compute $F$-signatures to be the volumes of various polytopes. Since the volume of an intersection of half-spaces is the same whether the half-spaces are open or closed, this technicality will not affect our arguments.)
\end{rem}

\section{Toric $F$-Signature Computation}\label{MainResultSection}

\begin{rem}(Conventions.)\label{ToricAssumptions}
For the remainder of this paper, $N$ is a lattice; $M = N^*$ is the dual lattice; $\sigma\subset N$ is a strongly convex rational polyhedral cone; $S = M\cap \sigma^{\dual}$, so that $k[S]$ is the coordinate ring of an affine toric variety in the notation of \cite{Ful93}, and $\vec{v}_1,\ldots, \vec{v}_r$ are primitive generators for $\sigma$. Unless stated otherwise, $k$ is perfect, and $\sigma$ is full-dimensional (i.e., the associated toric variety has no torus factors).
\end{rem}

\subsection{Statement of the main result, and an example}

\begin{defn}
Let $\sigma$ be a cone as in Remark \ref{ToricAssumptions}, with primitive generators $\vec{v}_1,\ldots, \vec{v}_r$. We define $P_{\sigma}\subset \sigma^{\dual}$ to be the polytope $\{\vec{w}\in M_{\mathbb{R}}\suchthat \forall i, 0\leq \vec{w}\cdot \vec{v}_i< 1\}$.
\end{defn}

\begin{thm}\label{AffineToricFSignature}
Let $R$ be the coordinate ring of an affine toric variety $X$ with no torus factors, with the conventions of Remark \ref{ToricAssumptions}. Then $s(R)$ is the volume of $P_{\sigma}$. More generally, suppose $X = X' \times T$, where $X'$ is a toric variety without torus factors and $T$ is an algebraic torus. Let $N'_{\mathbb{R}}\subset N_{\mathbb{R}}$ be the vector subspace spanned by $\sigma$, and let $\sigma'$ be $\sigma$ viewed as a cone in $N'_{\mathbb{R}}$. Then $s(k[X]) = s(k[X'])=\Vol(P_{\sigma'})$.
\end{thm}

We will prove this theorem in Section \ref{MainResultProofSection}. For now, we provide an example computation:

\begin{figure}[ht]\label{fig0}
\subfigure[The cone $\sigma\subset N_{\mathbb{R}}$.]{\label{QuadricConeSigma}
\includegraphics[scale=0.6]{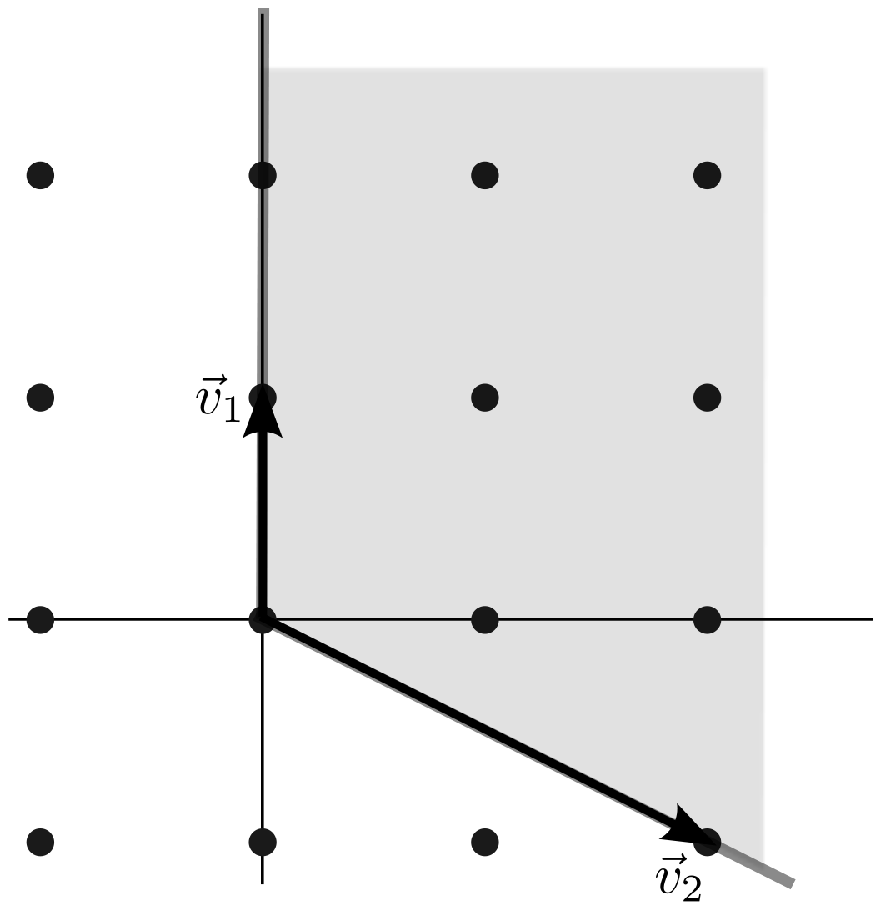}}
\hspace{10pt}
\subfigure[The dual cone $\sigma^{\dual}\subset M_{\mathbb{R}}$.]{\label{QuadricConePC}
\includegraphics[scale=0.6]{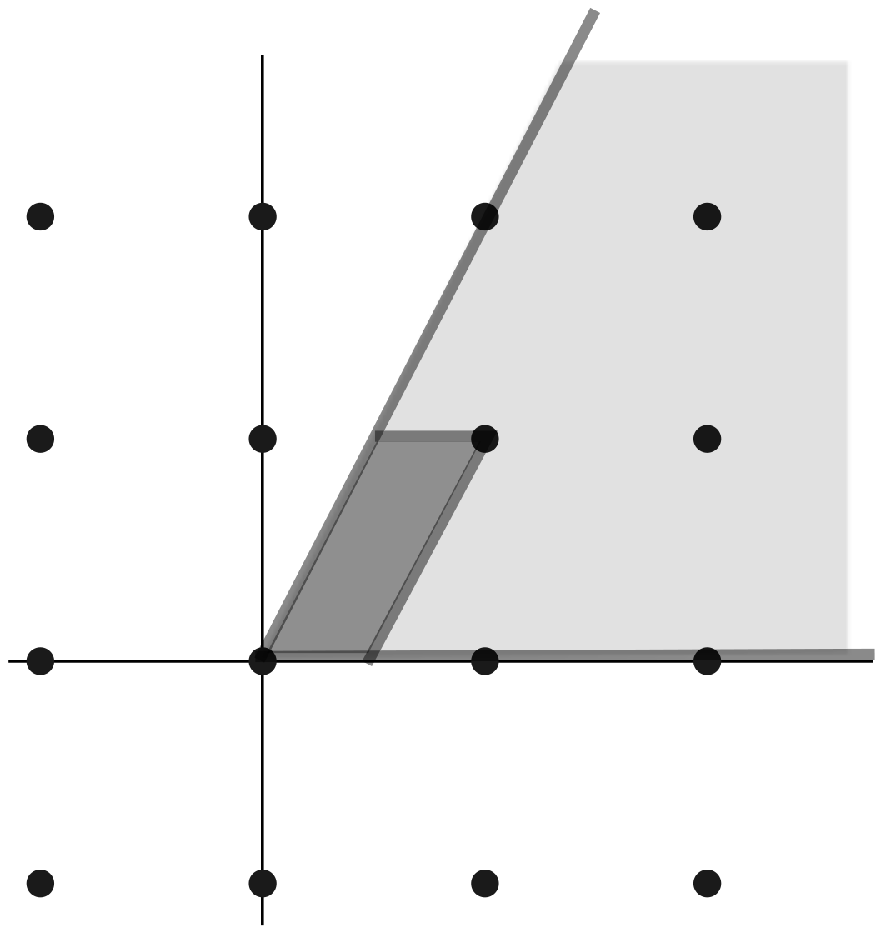}}
\caption{Computing the $F$-signature of the coordinate ring $k[x, xy, xy^2]$ of a quadric cone.}
\end{figure}

\begin{ex}
Figure \ref{QuadricConeSigma} shows the cone $\sigma$ corresponding to a plane quadric $\mathbb{V}(xy-z^2)$, with primitive generators $\vec{v}_1, \vec{v}_2$. Figure \ref{QuadricConePC} shows the dual cone $\sigma^{\dual}$. The coordinate ring $k[\sigma^{\dual}\cap M]$ is $k[x, xy, xy^2]$. In this case, $P_{\sigma}$, shaded in the figure, is the parallelogram $\{\langle x, y\rangle\suchthat 0\leq y<1, 0\leq 2x - y < 1\}$. The $F$-signature is $s(R) = \Vol(P_{\sigma}) = \frac{1}{2}$.
\end{ex}

\subsection{$R$-module Decomposition of $R^{1/q}$}

The main supporting result proved in this section is Lemma \ref{FreeRankFormula}, which gives a formula for the free rank of $R^{1/q}$ as an $R$-module in terms of the number of monomials in $R^{1/q}$ having a certain property. That lemma will be integral to our proof of the main theorem. Lemma \ref{FreeRankFormula} follows immediately from Lemma \ref{MonomialDecomposition}, which describes how $R^{1/q}$ decomposes as a direct sum of indecomposable $R$-module.

We will be able to compute the $F$-signature of a monomial ring $R$ because the $R$-module $R^{1/q}$ has an especially nice graded structure. In particular:

\begin{lem}
Let $R$ be a monomial ring, with $q = p^e$, and with character lattice $M\simeq\mathbb{Z}^n$. Then:
\begin{enumerate}
\item $R^{1/q}$ is finitely generated, as an $R$-module, by $q^{th}$ roots of monomials in $R$ of bounded degree.
\item $R^{1/q}$ admits a natural $\frac{1}{q}M$-grading which respects the $M$-grading on $R\subset R^{1/q}$. Each graded piece of $R^{1/q}$ a one-dimensional $k$-vector space.
\end{enumerate}
\end{lem}
\begin{proof}
\begin{enumerate}
\item If $R$ is generated by a finite set of monomials $\tau_i$, we can pick a minimal set of $t$ generators from among $\{\prod_i \tau_i^{a_i/q}\suchthat 0\leq a_i<q\}$.
\item We consider $R$ as a graded subring of the $M$-graded ring $k[M]$. The grading on $R^{1/q}$ is inherited in the obvious way: $\deg (\chi^m)^{1/q} = \frac{1}{q}\deg \chi^m = \frac{1}{q}m$. We conclude that $R^{1/q}$ has a natural $\frac{1}{q}M$-grading. Each graded piece consists of the set of $k$-multiples of a single monomial $\chi^{m/q}$.
\end{enumerate}
\end{proof}

It's well-known that relations on graded modules over monomial rings are generated by so-called ``binomial" relations. We supply a proof here, for lack of a better reference:

\begin{lem}
\label{MonomialModules} Let $W$ be a $G$-graded $R$-module, $G$ an abelian group, with each nonzero graded piece a one-dimensional $k$-vector space. (For example, when $R$ is a monomial subring of $k[x_1,\ldots, x_n]$, $W = R^{1/q}$ is $(\mathbb{Z}/q)^n$-graded.)
\begin{enumerate}
\item We can write $W$ as a quotient of a free module so that the relations are generated by ``binomial" relations, of the form $r\cdot \rho = s\cdot \mu$, for $r, s$ homogeneous elements of $R$ and $\rho, \mu$ homogeneous elements of $W$ such that $\deg r + \deg \rho = \deg s + \deg \mu$. 
\item We will say that two monomials $\rho, \mu\in W$ are \emph{related} if they satisfy a binomial relation. Then being related is an equivalence relation.
\end{enumerate}
\end{lem}
\begin{proof}
\begin{enumerate}
\item Let $\mu_i$ be graded generators for $W$. Let $\sum_i r_i \mu_i = 0$ be a relation. Since $W$ is graded, $\sum_i r_i \mu_i$ may be written as a sum of graded pieces, each of which is itself equal to 0. In other words, the relations on $W$ are generated by relations with the property that $r_i\mu_i$ has the same degree for each $i$. In that case, since each graded piece of $W$ is a one-dimensional $k$-vector space, we have that for each $i$, and each $j$ for which $r_i\mu_i\neq 0$, $r_i\mu_i = c_{ij} r_j\mu_j$ for some $c_{ij}\in k$. This is a binomial relation on $\mu_i$ and $\mu_j$, and binomial relations of this form generate the original relation $\sum_i r_i \mu_i = 0$.
\item If $r\rho = s\mu$, and $s'\mu = t\tau$, then $rs'\rho = st\tau$, so $\rho \sim \tau$.
\end{enumerate}
\end{proof}

The following lemma essentially indicates how to decompose $R^{1/q}$ as a direct sum of $R$-submodules generated by monomials. It also gives a condition describing which monomials generate free summands of $R^{1/q}$. 

\begin{lem}
\label{MonomialDecomposition}
Let $W$ be a finitely generated $G$-graded $R$-module, $G$ an abelian group, with each nonzero graded piece a one-dimensional $k$-vector space. (For example, $W = R^{1/q}$, $G = \mathbb{Z}^n$.) Let $H$ be a set of homogeneous generators for $W$. Let $A_1,\ldots, A_k\subset H$ be the distinct equivalence classes of elements of $H$ which are related (in the sense of Lemma \ref{MonomialModules}). Then:
\begin{enumerate}
\item $W \simeq \bigoplus_i R\cdot A_i$ is a direct sum of submodules generated by the sets $A_i$. 
\item Each of the submodules $R\cdot A_i$ is rank one (hence indecomposable even as an ungraded $R$-module). 
\item Finally, a homogeneous element $\mu\in W$ generates a free summand of $W$ if and only if the only homogeneous elements of $W$ that are related to $\mu$ are $R$-multiples of $\mu$.
\end{enumerate}
\end{lem}
\begin{proof}
\begin{enumerate}
\item Suppose that $A, B\subset H$, and $C = A \coprod B$. Since the corresponding submodules $R\cdot A$ and $R\cdot B$ are graded, their intersection must also be graded. In particular, in order for these modules to have nonempty intersection (equivalently, for the sum $R\cdot A+R\cdot B=R\cdot C$ to fail to be direct), we should have a binomial relation $r\mu = s\tau$ for some $\mu\in A, \tau\in B$, and $r, s\in R$, by Lemma \ref{MonomialModules}. However, we constructed the sets $A_i$ so that no binomial relations exist between them. We conclude that the sum is direct, and $W = R\cdot H \simeq \sum_i R\cdot A_i$.

\item Suppose that we have a subset $A\subset H$ of homogeneous elements which are are all related to one another (for example, $A = A_i$ for some $i$). Then pick a homogeneous $\mu\in A$. For any $\tau\in A$, we have that $r\mu = s\tau$ for $r, s\in R$; equivalently, $\tau = \frac{r}{s}\mu\in \Frac(R)\cdot\mu$. We conclude that $R\cdot A$ has rank 1, that is, $R\cdot A\otimes_R \Frac(R) \simeq \Frac(R)$.

\item Fix $\mu\in W$. Let $A_1,\ldots, A_k$ be the equivalence classes of related monomials in $W$. Without loss of generality, we may assume that $\mu \in A_1$. Then $R\cdot A_1$ is free of rank one if and only if it is generated by a single monomial. Thus, $R\cdot\mu$ is a free summand of $W$ if and only if $\mu$ generates $R\cdot A_1$, that is, if and only if $\mu$ divides every homogeneous element of $W$ that is related to $\mu$.
\end{enumerate}
\end{proof}


The following lemma will be essential in the next section when we compute the free rank $a_e$ of $R^{1/q}$ as an $R$-module. (We will also use this lemma in Section \ref{PairsAndTriplesSection} when computing the $F$-signature of pairs and triples.)

\begin{lem}\label{FreeRankFormula}
Let $R=k[S]$ be a monomial ring, $S$ a semigroup, and let $L = \Lattice(S)$ be the group generated by $S$. Fix $q=p^e$. Let $H\subset \frac{1}{q}L$ be a finitely generated $S$-module, so that $k[H]\subset k[\frac{1}{q}L]$ is an $R$-module finitely generated by monomials. Let $a_e$ be the free rank of $k[H]$ as an $R$-module. Then the set of monomials in $H$ which generate a free summand of $k[H]$ is $\{\chi^{\vec{v}}\suchthat \vec{v}\in H$, and $\forall \vec{k}\in L\backslash S, \vec{v}+\vec{k}\notin H\}$. Moreover, if 0 is the only unit in $H$, then $a_e$ is the size of this generating set.
\end{lem}

\begin{proof}
By Lemma \ref{MonomialDecomposition}, a monomial $\mu\in k[H]$ generates a free summand of $k[H]$ if and only if it is unrelated to all monomials in $k[H]$ that are not $R$-multiples of itself. We may characterize $\tau$ being related to $\mu$ (but not a multiple of it) by $\tau = \frac{r}{s}\mu$, with $\frac{r}{s}\in(\Frac{R})\backslash R$. Thus, the set of monomial generators for $k[H]$ which generate a free summand of $k[H]$ is $\{\mu\in \chi^H\suchthat$ for all monomials $\frac{r}{s}\in(\Frac{R})\backslash R, \frac{r}{s}\mu\notin k[H]\}$. We may rewrite this set as $\{\chi^{\vec{v}}\suchthat \vec{v}\in H$, and $\forall \vec{k}\in L\backslash S, \vec{v}+\vec{k}\notin H\}$, which is precisely the set described in the statement of the lemma. If 0 is the only unit in $H$, then there is a one-to-one-correspondence between monomials in the generating set and free summands of $W$. In that case, $a_e$ is the size of the generating set.
\end{proof}

\begin{rem}
As we'll see shortly, when we apply Lemma \ref{FreeRankFormula} to the case of the $R$-module $R^{1/q}$, the technical requirement that $0\in H$ be the only unit corresponds to the cone $\sigma$ being full-dimensional.
\end{rem}

\subsection{Proof of the Main Result}\label{MainResultProofSection}

\begin{rem}(An aside on computing volumes.)
Consider $M\subset M_{\mathbb{R}}$, a lattice abstractly isomorphic to $\mathbb{Z}^n$ contained in a vector space abstractly isomorphic to $\mathbb{R}^n$. Choosing a basis for $M$ gives us an identification of $M_{\mathbb{R}}$ with $\mathbb{R}^n$, hence a way to measure volume on $M_{\mathbb{R}}$. It's easily checked that this volume measure depends only on $M$ and not on our choice of basis for $M$. (Such a measure is uniquely determined by the fact that with respect to it, the measure of a fundamental parallelepiped for $M$, also called the \emph{covolume} of $M$, is 1.) Thus, it makes sense to talk about measuring volume ``relative to the lattice $M$," denoted $\Vol_M$ (or simply $\Vol$ when there is no risk of ambiguity).
\end{rem}

Now we are ready to prove our main result.

\begin{proof}[Proof of Theorem \ref{AffineToricFSignature}]

Suppose first that $X$ has no torus factors. We apply Lemma \ref{FreeRankFormula}, with $H = \sigma^{\dual}\cap\frac{1}{p^e}M$, $k[H] = R^{1/p^e}$. Since $\sigma$ is full-dimensional, $\sigma^{\dual}$ is strongly convex (by Lemma \ref{FullDimensionalToStronglyConvex}), so $H$ contains no nontrivial units. Then $$a_e = \#\{\vec{v}\in (\sigma^{\dual}\cap\frac{1}{p^e}M)\suchthat \forall \vec{k}\in M\backslash \sigma^{\dual}, \vec{v}+\vec{k}\notin \sigma^{\dual}\}.$$

Let $P^{'}_{\sigma}$ be the set $\{\vec{v}\in \sigma^{\dual}\suchthat \forall \vec{k}\in M\backslash \sigma^{\dual}, \vec{v}+\vec{k}\notin \sigma^{\dual}\}$.  Then $a_e = \#\{\vec{v}\in P^{'}_{\sigma}\cap \frac{1}{p^e}M\}$. By Lemma \ref{PCCharacterization}, $P^{'}_{\sigma} = P_{\sigma}$. Set $q = p^e$. Then $s(R)$, defined to be $\lim_{e\to\infty}\frac{a_e}{p^{en}}$, is equal to $\lim_{q\to\infty} \frac{\#(P_{\sigma}\cap \frac{1}{q}M)}{q^n}$. We apply Lemma \ref{PolytopeVolume} to conclude that $s(R) = \Vol(P_{\sigma})$. 

Suppose now that $X$ has torus factors. By Lemma \ref{ConeProduct}, $X \simeq X'\times T_{N''}$, where $X' = \Spec k[\sigma'\cap M']$ and $T_{N''}$ is the algebraic torus $\Spec k[M'']$. In particular, $R\simeq k[X']\otimes_k k[M'']$. We apply Theorem \ref{GradedProductFSignature} on the $F$-signature of products to see that $s(R) = s(k[X'])\cdot 1 = s(k[X']) =  \Vol(P_{\sigma'})$. (It is easy to check directly that $s(k[M''])=1$: writing $M''\simeq \mathbb{Z}^{d''}$, we see that $k[\mathbb{Z}^{d''}]^{1/q}$ is a free $k[\mathbb{Z}^{d''}]$-module of rank $q^{d''}$.)
\end{proof}

It remains to prove the two lemmas referenced in the proof of Theorem \ref{AffineToricFSignature}.

\begin{lem}\label{PCCharacterization}
Suppose that we are in the situation of Remark \ref{ToricAssumptions}. Then $$P^{'}_{\sigma} := \{\vec{v}\in \sigma^{\dual}\suchthat \forall \vec{k} \in M\backslash \sigma^{\dual}, \vec{v}+\vec{k}\notin \sigma^{\dual}\} = \{\vec{v}\in M_{\mathbb{R}}\suchthat \forall i, 0\leq \vec{v}\cdot \vec{v}_i<1\} =: P_{\sigma}.$$ 
\end{lem}
\begin{proof}
Recall that $\sigma^{\dual}=\{\vec{u}\suchthat\vec{u}\cdot \vec{v}_i\geq 0$ for all $i\}$.

Suppose $\vec{v}\in P_{\sigma}$, so that for each $i$, $0\leq \vec{v}\cdot\vec{v}_i<1$. Fix $\vec{k}\in M\backslash \sigma^{\dual}$. Since $\vec{k}\notin \sigma^{\dual}$, we know that $\vec{k}\cdot \vec{v}_j<0$ for some $j$. For such $j$, since $\vec{k}\cdot\vec{v}_j\in \mathbb{Z}$, we know that $\vec{k}\cdot\vec{v}_j\leq -1$. It follows that $(\vec{v}+\vec{k})\cdot\vec{v}_j<0$. Thus, $\vec{v}+\vec{k}\notin \sigma^{\dual}$. We conclude that $\vec{v}\in P^{'}_{\sigma}$. Hence, $P_{\sigma}\subset P^{'}_{\sigma}$.

Conversely, suppose that $\vec{v}\notin P_{\sigma}$, so that for some $j$, $\vec{v}\cdot \vec{v}_j\geq 1$. Set $\vec{k}_0$ to be any vector in $M$ such that $\vec{k}\cdot \vec{v}_j = -1$. (Such $\vec{k}_0$ exists since by Lemma \ref{ConeLattice}, $\Lattice(S) = M$.) Choose $\vec{k}_1$ to be any vector in $M$ that also lies in the interior of the facet $F_j = \vec{v}_j^{\perp}\cap \sigma^{\dual}$ of $\sigma^{\dual}$. Then $\vec{k}_1\cdot \vec{v}_i>0$ for each $i\neq j$. Thus, for sufficiently large $m$, $(\vec{k}_0+m\vec{k}_1)\cdot \vec{v}_i \geq 0$ for $i\neq j$, while $(\vec{k}_0+m\vec{k}_1)\cdot \vec{v}_j = 0$. Set $\vec{k} = \vec{k}_0+m\vec{k}_1$. Then $\vec{k}\in M$, but $\vec{v}\notin \backslash \sigma^{\dual}$, since $\vec{k}\cdot \vec{v}_j=-1<0$. On the other hand, $\vec{v}+\vec{k}\in \sigma^{\dual}$, since $(\vec{v}+\vec{k})\cdot \vec{v}_i\geq 0$ for each $i$. We conclude that $\vec{v}\notin P^{'}_{\sigma}$.

Hence, $P^{'}_{\sigma}\subset P_{\sigma}$. We conclude that $P_{\sigma} = P^{'}_{\sigma}$, as we desired to show.
\end{proof}


\begin{lem}\label{PolytopeVolume}
Let $M$ be a lattice and $P\subset M\otimes_{\mathbb{Z}}\mathbb{R}$ a polytope (or, more generally, any set whose boundary has measure zero). Then $\lim_{q\to\infty} \frac{\#\{P\cap \frac{1}{q}M\}}{q^n} = \Vol(P)$.
\end{lem}
\begin{proof}
In fact, when $P$ is a polytope, it can be shown that the quantity $\#\{P\cap \frac{1}{q}M\}$ is polynomial in $q$ of degree $n$, known as the \emph{Ehrhart polynomial} of $P$, and that its leading coefficient $\lim_{q\to\infty} \frac{\#\{P\cap \frac{1}{q}M\}}{q^n}$ is $\Vol(P)$  (\cite{MS05}, Thm 12.2). Even without this fact, however, it is easy to sketch a proof of the special case that we require: the quantity $\lim_{q\to\infty} \frac{\#\{P\cap \frac{1}{q}M\}}{q^n}$ is a limit of Riemann sums measuring the volume of $P$ with respect to the lattice $M$. (See, for example, \cite{Fol99}, Theorem 2.28.)
\end{proof}

\section{$F$-Signature of Pairs and Triples}\label{PairsAndTriplesSection}

\subsection{Definitions}

The notion of singularities of pairs is an important one in birational geometry. Instead of studying the singularities of a variety $X$, or the singularities of a divisor $D$ on $X$, one studies the pair $(X, D)$, for example by considering how $D$ changes under various resolutions of $X$. (See \cite{Kol97} for an introduction to pairs in this setting.) More generally, it is often useful to study triples $(X, D, \mathfrak{a})$, where $D$ is a divisor and $\mathfrak{a}$ an ideal sheaf on $X$. 

The $F$-signature of pairs and triples was recently defined in \cite{BST11}. First, we recall:

\begin{defn}
Let $R$ be a normal domain. Let $D$ be an effective Weil divisor on $X=\Spec R$. We define $R(D)$ to be the module of global sections of $O_X(D)$. That is, $R(D) = \{f\in \Frac(R)\suchthat \divf f + D\geq 0\}$. 
\end{defn}

\begin{rem}
Note that $R\subset R(D)$, and that if $D$ is the principal divisor $\divf g$ for some $g\in R$, then $R(D) = R\cdot \frac{1}{g}$ is the cyclic $R$-module generated by $\frac{1}{g}$.
\end{rem}

Now we can define the $F$-signature of a pair or triple:

\begin{defn}\label{TriplesDef}
Let $(R, m)$ be a normal local (or $\mathbb{N}$-graded) domain over $k$, of dimension $d$. Let $D=\sum_i a_i D_i$ be an effective $\mathbb{Q}$-divisor on $X=\Spec R$, $\mathfrak{a}$ an ideal of $R$, and $0\leq t \in \mathbb{R}$. We define the $F$-signature of the triple $(R, D, \mathfrak{a}^t)$ as follows. For each $e$, let $\mathfrak{C}_e = \Hom_R(R^{1/p^e}, R)$. Let $\mathfrak{D}_e$ be the $R^{1/p^e}$-submodule of $\mathfrak{C}_e$ defined as $\Hom_R(R(\lceil (p^e-1)D\rceil)^{1/p^e}, R)$. (The module structure is given by premultiplication, $r^{1/p^e}\cdot\phi(x) = \phi(r^{1/p^e} x)$.) Define $I_e^D\subset R$ to be the ideal $\{r\in R\suchthat\forall \phi\in\mathfrak{D}_e, \phi(r^{1/p^e})\in m\}$. Then the $F$-signature $s(R, D)$ is $\lim_{e\to\infty} \frac{l(R/I_e^D)}{p^{ed}}$.

Let $\mathfrak{D}'_e =\mathfrak{D}_e\cdot(F_*^e(\mathfrak{a}^{\lceil(p^e-1)t\rceil}))$, the $R^{1/p^e}$-submodule of $\mathfrak{D}_e$ generated by $\{[x\mapsto\phi(a^{1/p^e}\cdot x)]\suchthat a\in \mathfrak{a}^{\lceil(p^e-1)t\rceil}, \phi\in \mathfrak{D}_e\}$. 

Define $I_e^{\mathfrak{a}}\subset R$ to be the ideal $\{r\in R\suchthat \forall \phi\in\mathfrak{D}'_e, \phi(r^{1/p^e})\in m\}$. The $F$-signature $s(R, D, \mathfrak{a}^t)$ is defined to be $\lim_{e\to\infty} \frac{l(R/I_e^{\mathfrak{a}})}{p^{ed}}$.
\end{defn}

\begin{rem}
The limits given in Definition \ref{TriplesDef} have been shown to exist in \cite{BST11}, in the case of a local ring.
\end{rem}

\begin{rem}
It's easily checked that the $F$-signature of the triple $(R, D, (1))$ (with $\mathfrak{a}$ as the unit ideal) is the $F$-signature of the pair $(R, D)$. Likewise, the $F$-signature of the pair $(R, 0)$ (with $D$ as the zero divisor) is the $F$-signature of $R$.
\end{rem}

Just like the ``usual" $F$-signature, the $F$-signature of pairs may be viewed as a measure of the number of splittings of the Frobenius map, or as a measure of the number of free summands splitting off from $R^{1/p^e}$, though the $F$-signature of pairs only counts certain summands:

\begin{lem} \label{PairsInterp}(See \cite{BST11}, Proposition 3.5.)
Suppose that we are in the setting of Definition \ref{TriplesDef}. Set $a_e^D = l(F_*(R/I_e^D))$, so that $s(R, D) = \lim_{e\to\infty} \frac{a_e^D}{p^{e(d+\alpha)}}$. Then $a_e^D$ is the maximum rank of a free summand of $R(\lceil (p^e-1)D\rceil)^{1/p^e}$ that is simultaneously a free summand of the submodule $R^{1/p^e}$. Moreover, any $k$-vector space basis for $R^{1/p^e}/(I_e^D)^{1/p^e}$ lifts to a set of generators in $R^{1/p^e}$ for such a free summand of maximum rank, and $(I_e^D)^{1/p^e}$ is the submodule of elements of $R^{1/p^e}$ which do not generate such a free summand.
\end{lem}

Minor modifications can be made to $\mathfrak{D}_e$ or to $\mathfrak{D}'_e$ in the definition, without changing the $F$-signature. In particular:

\begin{lem}(\cite{BST11}, Lemma 4.17.)
Let $R, D, \mathfrak{a}, t$ be as in Definition \ref{TriplesDef}. Suppose that for each $e$, $\mathfrak{D}_e$ is replaced by some $\mathfrak{D}'_e\subset\mathfrak{C}_e$ such that for some $0\neq c\in R$, $c^{1/p^e}\mathfrak{D}_e\subset \mathfrak{D}'_e$ and $c^{1/p^e}\mathfrak{D}'_e\subset \mathfrak{D}_e$. Then the limits given in Definition \ref{TriplesDef} are unchanged by this replacement.
\end{lem}

This lemma allows us to make many simplifications in our $F$-signature computations:

\begin{lem}\label{ApproxTriples}(\cite{BST11}, discussion following Lemma 4.17.)
Let $R, D, \mathfrak{a}, t$ be as in Definition \ref{TriplesDef}.

Suppose that for some sequence of divisors $D'_e$, the coefficients of $(p^e-1)D-D'_e$ are bounded. Then replacing $\mathfrak{D}_e = Hom_R(R(\lceil (p^e-1)D\rceil)^{1/p^e}, R)$ with $Hom_R(R(\lceil D'_e\rceil)^{1/p^e}, R)$ in Definition \ref{TriplesDef} does not change $s(R, D, \mathfrak{a}^t)$. (In particular, we may replace $(p^e-1)D$ by $p^eD$ in the limits given in Definition \ref{TriplesDef} without changing the $F$-signature.) If we write $D = \sum_i a_i D_i$, then $s(R, D, \mathfrak{a}^t)$ is continuous in $a_i$.

Likewise, suppose $\lceil(p^e-1)t\rceil - t'_e$ is bounded for some sequence of exponents $t'_e$. Then replacing $\lceil(p^e-1)t\rceil$ with $\lceil t'_e\rceil$ in the definition does not change $s(R, D, \mathfrak{a}^t)$. Also, $s(R, D, \mathfrak{a}^t)$ is continuous in $t$.

Finally, replacing $\mathfrak{a}^{t'_e}$ with its integral closure $\overline{\mathfrak{a}^{t'_e}}$ does not change the $F$-signature.
\end{lem}

Lemma \ref{ApproxTriples} will allow us to simplify our computations later by, for example, assuming that $t$ and the coefficients $a_i$ of $D$ are rational numbers with denominator a power of $p$.

We also have, in the triples case:

\begin{lem}\label{TriplesInterp}
Suppose that we are in the setting of Definition \ref{TriplesDef}. Then $I_e^{\mathfrak{a}} = (I_e^D:\mathfrak{a}^{\lceil t(p^e-1)\rceil})$. Equivalently, set $a_e^{\mathfrak{a}} = l(F^e_*R/F^e_*I_e^{\mathfrak{a}})$, so that $s(R, D, \mathfrak{a}^t) = \lim_{e\to\infty} \frac{a_e^{\mathfrak{a}}}{p^{e(d+\alpha)}}$. Then $I_e^{\mathfrak{a}} = (I_e^D:_R \mathfrak{a}^{\lceil t(p^e-1)\rceil})$, and $$a_e^{\mathfrak{a}} = l(F^e_*R/(F_*^e I_e^D:_{R^{1/p^e}} F_*^e\mathfrak{a}^{\lceil t(p^e-1)\rceil})).$$
\end{lem}
\begin{proof}
Requiring that $r\in (I_e^D:\mathfrak{a}^{\lceil t(p^e-1)\rceil})$ is the same as requiring that multiplication by an element $a\in (\mathfrak{a}^{\lceil t(p^e-1)\rceil})^{1/p^e}$ sends $r^{1/p^e}$ into $(I_e^D)^{1/p^e}$, so that for all $\phi\in\mathfrak{D}_e$, $\phi(a\cdot r^{1/p^e})\in m$. This is equivalent to saying that for all $\phi\in\mathfrak{D}'_e, \phi(r^{1/p^e})\in m$. We conclude that $r\in (I_e^D:\mathfrak{a}^{\lceil t(p^e-1)\rceil})\iff r\in I_e^{\mathfrak{a}}$. We have proved our first claim; the second claim follows immediately.
\end{proof}

\subsection{Toric preliminaries for pairs and triples}

For our pair and triple computations, we will require some understanding of divisors on toric varieties. (Unless stated otherwise, proofs of these results may be found in \cite{Ful93}).

\begin{defn}
A prime Weil divisor $D$ on a toric variety $X$ is \emph{torus-invariant} if it is invariant under the action of the embedded torus on $X$. More generally, a divisor $D$ is torus-invariant if $D = \sum_i a_i D_i$, where $D_i$ are the torus-invariant prime divisors of $X$.
\end{defn}

The torus-invariant prime divisors of $X$ are in bijective correspondence with primitive generators $\vec{v}_i$ of the cone $\sigma\subset N_{\mathbb{R}}$ of $X$. (In particular, the prime divisor corresponding to $\vec{v}_i$ is $D_i = \mathbb{V}(I_i)$, where $I_i$ is the ideal generated by monomials $\vec{u}$ such that $\vec{u}\cdot \vec{v}_i \neq 0$.) It can be shown that $\nu_i: \Frac(R)\to\mathbb{Z}$, the discrete valuation corresponding to $D_i$, is given by $\nu_i(\vec{u}) = \vec{u}\cdot \vec{v}_i$. From this, it follows that:

\begin{lem}\label{RTwist}
Let $X = \Spec R$ be an affine toric variety, $R = k[S]$, $S = \sigma^{\dual}\cap M\subset M_{\mathbb{R}}$. Let $D=\sum_i a_i D_i$ be a torus-invariant divisor on $X$, where each $D_i$ corresponds to a primitive generator $\vec{v}_i$ of $\sigma$. Then $R(D) = \sum_{\vec{u}} R\cdot x^{\vec{u}}$, where the sum is taken over all $\vec{u}\in S$ such that $\vec{u}\cdot \vec{v}_i\geq -a_i$.
\end{lem}

For our $F$-signature of triples computation, we will require the concept of the Newton polyhedron of a monomial ideal.

\begin{defn}
A \emph{polyhedron} is a possibly unbounded intersection of finitely many half-spaces.
\end{defn}

\begin{defn}
Let $R = k[S]$ be a monomial ring, with $S\subset M_{\mathbb{R}}$, as above. Let $\mathfrak{a}\subset R$ be a monomial ideal (i.e., an ideal generated by monomials). The \emph{Newton polyhedron} of $\mathfrak{a}$, a polyhedron in $M_{\mathbb{R}}$, is the convex hull of the set of monomials in $\mathfrak{a}$.
\end{defn}

The Newton polyhedron is closely related to the integral closure of monomial ideals:

\begin{lem} (\cite{Vil01}, Proposition 7.3.4)
Let $R = k[S]$ be a monomial ring as above. Let $\mathfrak{a}\subset R$ be a monomial ideal. Then the integral closure $\overline{\mathfrak{a}}$ of $\mathfrak{a}$ in $R$ is a monomial ideal generated by those monomials in the set $\Newt(\mathfrak{a})\cap M$.
\end{lem}

\begin{defn}
Let $Q_1, Q_2$ be subsets of $\mathbb{R}^n$. The \emph{Minkowski sum} of $Q_1$ and $Q_2$, denoted $Q_1 + Q_2$, is the set $\{\vec{u}_1 + \vec{u}_2\suchthat \vec{u}_1\in Q_1, \vec{u}_2\in Q_2\}$. We denote by $Q_1 - Q_2$ the set $Q_1 + (-Q_2) = \{\vec{u}_1 - \vec{u}_2\suchthat \vec{u}_1\in Q_1, \vec{u}_2\in Q_2\}$.
\end{defn}

It's easy to see that the Minkowski sum of two polyhedrons is itself a polyhedron, and that the sum of two polytopes is a polytope. (See, for example, \cite{Gru03}, \S 15.1.)

As Corollary \ref{GorensteinTriples} is a statement about $\mathbb{Q}$-Gorenstein pairs, we recall the definition of the $\mathbb{Q}$-Gorenstein condition.

\begin{defn}If $D$ is an effective divisor on $X = \Spec R$, the pair $(R, D)$ is \emph{$\mathbb{Q}$-Gorenstein} if, fixing a canonical divisor $K_X$ on $X$, the divisor $K_X+D$ is $\mathbb{Q}$-Cartier; that is, some integer multiple of $K_X+D$ is Cartier.
\end{defn}

It happens that on a toric variety, a canonical divisor may be given by $K_X = -\sum_i D_i$, where the sum is taken over all torus-invariant divisors on $X$. It's also a fact that Cartier divisors on an affine toric variety are principal. We conclude:

\begin{lem}\label{QGorensteinPair}
Let $X = \Spec k[S]$ be an affine toric variety with a corresponding strongly convex polyhedral cone $\sigma$. Let $D = \sum_i a_i D_i$ be a divisor. Then $(R, D)$ is $\mathbb{Q}$-Gorenstein if and only if for some $\vec{w}\in M\otimes\mathbb{Q}$, $\vec{w}\cdot\vec{v}_i = -1+a_i$ for each $i$.
\end{lem}
\begin{proof}
Given $\vec{u}\in M$, $\divf x^{\vec{u}} = \sum_i (\vec{u}\cdot\vec{v}_i)D_i$. This operation extends linearly to $\mathbb{Q}$-divisors, so that for $\vec{u}\in M\otimes_{\mathbb{Z}}\mathbb{Q}$, $\divf x^{n\vec{u}} = n(\sum_i c_i D_i)$ if and only if $\vec{u}\cdot \vec{v}_i = c_i$ for each $i$. Thus, $K_X+D$ is $\mathbb{Q}$-Gorenstein if and only if for some $\vec{w}\in M\otimes\mathbb{Q}$, $\vec{w}\cdot\vec{v}_i = -1 + a_i$ for each $i$.
\end{proof}

\subsection{Pairs Computation}

Now we will compute the $F$-signature of pairs and triples. We begin with the pairs case, in which our proof requires little modification from that of Theorem \ref{AffineToricFSignature}.

\begin{defn}\label{PDDef}
Let $\sigma$ be a cone as in Remark \ref{ToricAssumptions}, with primitive generators $\vec{v}_1,\ldots, \vec{v}_r$. Let $D = \sum_i a_i D_i$ be a torus-invariant $\mathbb{Q}$-divisor on $\Spec R$, with $D_i$ the prime divisor corresponding to $\vec{v}_i$. We define $P_{\sigma}^D\subset \sigma^{\dual}$ to be the polytope $\{\vec{v}\in M_{\mathbb{R}}\suchthat \forall i, 0\leq \vec{v}\cdot \vec{v}_i<1-a_i\}$.
\end{defn}

\begin{thm}\label{FSignatureOfPairs}
Let $R$ be the coordinate ring of an affine toric variety, with conventions as in Remark \ref{ToricAssumptions}. Let $D$ be a torus-invariant $\mathbb{Q}$-divisor, with associated polytope $P_{\sigma}^D$ as in Definition \ref{PDDef}. Then $F_*^e I_e^D$ is generated by the monomials in the set $(\sigma\backslash P_{\sigma}^D)\cap \frac{1}{p^e}M$, and $s(R) = \Vol(P_{\sigma}^D)$.
\end{thm}

\begin{proof}
First, we apply Lemma \ref{ApproxTriples} to replace $(p^e-1)D$ by $p^eD$ without changing the $F$-signature. By the same lemma, $s(R, D)$ is continuous as a function of the $a_i$, so we may assume that $a_i\in\frac{1}{p^e}\mathbb{Z}$. (Proving the claim on that dense subset will prove it for all divisors by continuity.) We also assume that $e$ is sufficiently large, so that $p^eD$ is an integral divisor, and $\lceil p^eD\rceil = p^eD$. As a result of all this simplification, we may ignore the rounding-up operation.

By Lemma \ref{RTwist}, $R(p^eD)$ is an $\mathbb{N}^n$-graded $R$-module, generated by $\{\chi^{\vec{v}}\suchthat\vec{v}\in M$, and $\vec{v}\cdot \vec{v}_i\geq -p^e a_i\}$. It follows that $R(p^eD)^{1/p^e}$ is $\mathbb{N}^n/q$-graded: it's generated by $\{\chi^{\vec{v}}\suchthat\vec{v}\in \frac{1}{p^e}M$, and $\vec{v}\cdot \vec{v}_i\geq -a_i\}$.

Thus, we may apply Lemma \ref{MonomialDecomposition}.  The graded module $R(p^eD)^{1/p^e}$ decomposes as a direct sum of graded submodules, where each submodule is generated by related monomials. Each submodule splits off from $R(p^eD)^{1/p^e}$ if and only if it is generated by a single monomial; likewise, each submodule generated by monomials in $R^{1/p^e}$ splits off from $R^{1/p^e}$ if and only if it is generated by a single monomial.

Set $\sigma' = \{\vec{v}\in M_{\mathbb{R}}\suchthat \forall i, \vec{v}\cdot \vec{v}_i\geq -a_i\}$, so that if $S' := M\cap \sigma'$, then $\chi^{S'}$ is the set of generators for $R(D)$. By Lemma \ref{FreeRankFormula}, $a_e = \#\{\vec{v}\in \frac{1}{p^e}S\suchthat \forall \vec{k}\in M\backslash S, \vec{v}+\vec{k}\notin \frac{1}{p^e}S'\}$. Note that $F_*^e I_e^D$ is generated by those monomials in $R^{1/p^e}$ whose corresponding characters are \emph{not} in this set.

 Following the proof of Theorem \ref{AffineToricFSignature}, we find that $a_e = \#\{\vec{v}\in \frac{1}{p^e}S\suchthat \forall \vec{k}\in M\backslash S, \vec{v}+\vec{k}\notin \frac{1}{p^e}S'\}$. Equivalently, $a_e = \#\{\vec{v}\in \sigma^{\dual}\cap \frac{1}{p^e}M\suchthat \forall \vec{k}\in M\backslash \sigma^{\dual}, \vec{v}+\vec{k}\notin \sigma'\}.$ That is, $$a_e = \#\{\frac{1}{p^e}M\cap P'\},$$ where $P'=\{\vec{v}\suchthat \vec{v}\cdot\vec{v}_i\geq 0$, and $\forall \vec{k}\in M\backslash\sigma^{\dual}, (\vec{v}+\vec{k})\cdot \vec{v}_i<-a_i$ for some $i\}$. (By the same argument, $F_*^e I_e$ is generated by the monomials whose characters lie in $\sigma^{\dual}\backslash P'$.) By Lemma \ref{PCDCharacterization} (the pairs analogue to Lemma \ref{PCCharacterization}), $P' = P_{\sigma}^D$. Our claim then follows from our lemma on volumes of polytopes, Lemma \ref{PolytopeVolume}, just as in our original proof of Theorem \ref{AffineToricFSignature}.
\end{proof}

\begin{lem}\label{PCDCharacterization}
Suppose that we are in the situation of Lemma \ref{FSignatureOfPairs}, and that $P' = \{\vec{v}\suchthat \vec{v}\cdot\vec{v}_i\geq 0$, and $\forall \vec{k}\in M\backslash \sigma^{\dual},(\vec{v}+\vec{k})\cdot \vec{v}_i<-a_i$ for some $i\}$. Then $P' = P_{\sigma}^D = \{\vec{v}\in M_{\mathbb{R}}\suchthat \forall i, 0\leq \vec{v}\cdot \vec{v}_i<1-a_i\}$.
\end{lem}
\begin{proof}
We follow the proof of Lemma \ref{PCCharacterization}. Suppose that $\vec{v}\in P_{\sigma}^D$, $\vec{k}\in M$, and $\vec{k}\cdot \vec{v}_i<0$. Then $(\vec{v}+\vec{k})\cdot \vec{v}_i< (1-a_i)+(-1)= -a_i$, so $\vec{v}+\vec{k}\notin \sigma'$. It follows that $P_{\sigma}^D\subset P'$. On the other hand, suppose $\vec{v}\notin P_{\sigma}^D$. Either $\vec{v}\cdot \vec{v}_i<0$ for some $i$, in which case $\vec{v}\notin P'$, or $\vec{v}\cdot \vec{v}_i\geq 1-a_i$ for some $i$. In the latter case, we may, as in Lemma \ref{PCCharacterization}, choose $\vec{k}\in M$ such that $\vec{k}\cdot \vec{v}_i=-1$ and $\vec{k}\cdot\vec{v}_j\geq 0$ for all $j\neq i$. Then $\vec{k}\notin \sigma^{\dual}$, but $(\vec{v}+\vec{k})\cdot\vec{v}_j\geq -a_j$ for all $j$. It follows that $\vec{v}\notin P'$.

We conclude that $P' = P_{\sigma}^D$.
\end{proof}

\subsection{Triples Computation}

\begin{defn}\label{PaDef}
Let $\sigma$ be a cone as in Remark \ref{ToricAssumptions} and $D$ a torus-invariant divisor, with corresponding polytope $P_{\sigma}^D$ as in Definition \ref{PDDef}. Let $\mathfrak{a}\subset R$ be a monomial ideal, and $0\leq t\in \mathbb{R}$. Let $\Newt(\mathfrak{a})$ denote the Newton polyhedron of $\mathfrak{a}$. We define $\PDa$ to be the polytope $(P_{\sigma}^D-t\cdot \Newt(\mathfrak{a}))\cap \sigma^{\dual}$.
\end{defn}

\begin{thm}\label{FSignatureOfTriples}
Let $R$ be the coordinate ring of an affine toric variety, with conventions as in Remark \ref{ToricAssumptions}. Let $D$ be a torus-invariant divisor as in Definition \ref{PDDef}. Let $\mathfrak{a}\subset R$ be a monomial ideal, with associated polytope $\PDa$ as in Definition \ref{PaDef}. Then $s(R, D, \mathfrak{a}^t) = \Vol(\PDa)$.
\end{thm}

\begin{proof}[Proof of Theorem \ref{FSignatureOfTriples}]
As in our proof of Lemma \ref{FSignatureOfPairs}, we apply Lemma \ref{ApproxTriples} to replace $(p^e-1)D$ by $p^eD$, and to assume that $a_i\in\frac{1}{p^e}\mathbb{Z}$. Likewise, we assume that $t\in \frac{1}{p^e}\mathbb{Z}$, and we replace $(p^e-1)t$ with $p^e t$, so that for sufficiently large $e$, $\lceil p^e t\rceil = p^e t$, and $\lceil p^e D\rceil = p^e D$. We also replace $\mathfrak{a}^{p^e t}$ with its integral closure $\overline{\mathfrak{a}^{p^e t}}$, which is generated by monomials in the set $p^e t\cdot\Newt(\mathfrak{a})$.

We will use the characterization of $F$-signature of triples given in Lemma \ref{TriplesInterp}. Thus, we study $(I_e^{\mathfrak{a}})^{1/p^e} = ((I_e^D)^{1/p^e}:(\overline{\mathfrak{a}^{p^e t}})^{1/p^e})$. By Lemma \ref{FSignatureOfPairs}, $(I_e^D)^{1/p^e}$ is generated by the monomials whose characters lie in $(\sigma^{\dual}\backslash P_{\sigma}^D)\cap \frac{1}{p^e}M$. The set of characters $\vec{v}$ with $\chi^{\vec{v}}\in (\overline{\mathfrak{a}^{tp^e}})^{1/p^e}$ is $(t\cdot \Newt(\mathfrak{a}))\cap \frac{1}{p^e}M$. Thus, the monomials in $R^{1/p^e}\backslash (I_e^{\mathfrak{a}})^{1/p^e}$ are those $\chi^{\vec{v}},\vec{v}\in\frac{1}{p^e}M\cap \sigma^{\dual}$, such that for some $\vec{w}\in \frac{1}{p^e}M\cap t\cdot \Newt(\mathfrak{a}), \vec{v}+\vec{w}\in P_{\sigma}^D$. This set of characters can be written as a Minkowski sum, so that the size $a_e^{\mathfrak{a}}$ of the set is: $$a_e^{\mathfrak{a}} = \#(((P_{\sigma}^D\cap \frac{1}{p^e}M)-((t\cdot \Newt(\mathfrak{a}))\cap \frac{1}{p^e}M))\cap \sigma^{\dual}).$$

We obtain a slightly larger (but easier-to-count) set if we intersect with the lattice $\frac{1}{p^e}M$ only \emph{after} taking the Minkowski sum. In particular, set $a'_e := \#( (P_{\sigma}^D - t\cdot\Newt(\mathfrak{a}))\cap \sigma^{\dual} \cap \frac{1}{p^e}M )$. Note that $a'_e = \#(\PDa\cap \frac{1}{p^e}M)$. Now, $a'_e$ may be larger than $a^{\mathfrak{a}}_e$. However, by Lemma \ref{TriplesErrorCorrect}, $\lim_{e\to\infty} \frac{a_e^{\mathfrak{a}}}{p^{ed}} = \lim_{e\to\infty} \frac{a'_e}{p^{ed}}$. 

Thus, $s(R, D, \mathfrak{a}^t) = \lim_{e\to\infty} \frac{a'_e}{p^{ed}}$. We can apply Lemma \ref{PolytopeVolume} (with $M = \mathbb{Z}^n$, $P = \PDa$, and $a'_e = \#(P\cap \frac{1}{p^e}M)$) to conclude that the $F$-signature of triples is the volume of the polytope $\PDa$.
\end{proof}






\subsection{A Technical Lemma}

All that remains is to prove Lemma \ref{TriplesErrorCorrect}, which states that in the proof of Theorem \ref{FSignatureOfTriples}, the quantities $a^{\mathfrak{a}}_e$ and $a'_e$ are ``close enough" that either one may be used to compute $F$-signature. (These two quantities are obtained similarly: to compute $a^{\mathfrak{a}}_e$, we start with the polytopes $P_{\sigma}^D$ and $-t\cdot\Newt(\mathfrak{a})$; intersect each with the lattice $\frac{1}{p^e}M$; then take the Minkowski sum of these two sets. To compute $a'_e$, we take the Minkowski sum of the two polytopes, then intersect with the lattice $\frac{1}{p^e}M$.)

\begin{lem}\label{TriplesErrorCorrect}
Suppose we are in the situation of Theorem \ref{FSignatureOfTriples}. Assume that all coefficients (the $a_i$ and $t$) lie in $\frac{1}{p^{e_0}}M$ for some $e_0$. Let $a_e^{\mathfrak{a}} = \#((\frac{1}{p^e}M\cap P_{\sigma}^D-\frac{1}{p^e}M\cap (t\cdot \Newt(\mathfrak{a}))\cap P_{\sigma}^D))$ and $a'_e = \#((P_{\sigma}^D-((t\cdot \Newt(\mathfrak{a}))\cap P_{\sigma}^D))\cap\frac{1}{p^e}M\cap P_{\sigma}^D)$. Then $\lim_{e\to\infty}\frac{a'_e-a_e^{\mathfrak{a}}}{p^{ed}} = 0$.
\end{lem}

To prove Lemma \ref{TriplesErrorCorrect}, we wish to show that taking the Minkowski sum of two polytopes, then intersecting with a lattice, is ``roughly the same" as performing those operations in reverse order. That is the content of Lemma \ref{SetDifferences}, the proof of which will proceed in several short steps. (Lemma \ref{PolytopeLatticePoints} will be used to prove Lemma \ref{CloseSets}, and Lemmas \ref{CloseSets}, \ref{HalfSpaceComplement} and \ref{BallInHalfSpace} will be used to prove Lemma \ref{SetDifferences}. Lemmas \ref{CloseSets} and \ref{SetDifferences} will then be used to prove Lemma \ref{TriplesErrorCorrect}.)

\begin{rem}[Notation]
In what follows, let $M$ be a lattice, and $M_{\mathbb{R}} = M\otimes\mathbb{R}\simeq \mathbb{R}^n$. Fix $e_0>0$, and set $M' = \frac{1}{p^{e_0}}M$. We denote by $d(\vec{v}, U)$ the distance from a point $\vec{v}$ to a set $U$, given by $d(\vec{v}, U)=\inf_{\vec{u}\in U}d(\vec{v}, \vec{u})$. We denote by $\mathcal{B}(\vec{v}, r)$ the ball of radius $r$ around $\vec{v}$, given by $\{\vec{u}\suchthat d(\vec{v}, \vec{u})<r\}$. We denote by $\partial U$ the boundary of a set $U$. Set $\lfloor x\rfloor_e = \frac{\lfloor p^e x\rfloor}{p^e}$ (rounding down to the nearest multiple of $\frac{1}{p^e}$) and $[x]_e = x-\lfloor x\rfloor_e$ (the $\frac{1}{p^e}$-fractional part of $x$).
\end{rem}

\begin{lem}
\label{PolytopeLatticePoints}Fix any polytope $Q\subset M_{\mathbb{R}}$ with extremal points in $M'$. There is a constant $K$ such that for each $e$, and for $\vec{u}\in Q$, $d(\vec{u}, Q\cap \frac{1}{p^e}M)<\frac{K}{p^e}$.
\end{lem}
\begin{proof}
Let $\vec{u}_1,\ldots, \vec{u}_k$ denote the extremal points of $Q$; all are contained in $\frac{1}{p^{e_0}}M$. Fix $\vec{v}\in Q$, say $\vec{v} = \sum_i a_i\vec{u_i}$ with $0\leq a_i; \sum_i a_i = 1$. Note that $\sum_i [a_i]_e = 1-\sum_i\lfloor a_i\rfloor_e \in\frac{1}{p^e}\mathbb{Z}$. Furthermore, $\sum_i [a_i]_e< \frac{k}{p^e}$ (each of the $k$ terms in the sum is less than $\frac{1}{p^e}$). Suppose without loss of generality that $\vec{u}_1$ has the greatest length of any of the extremal points.

Set $\vec{u}^* = \sum_i \lfloor a_i\rfloor_{e-e_0} \vec{u}_i + (\sum_i [a_i]_{e-e_0})\vec{u}_1$. Then $\vec{u}^*\in Q\cap \frac{1}{p^e}M$. Moreover, $d(\vec{u},\vec{u}^*)\leq \sum_i [a_i]_{e-e_0}|\vec{u}_i| + (\sum_i [a_i]_{e-e_0})|\vec{u}_1|\leq \frac{2k}{p^{e-e_0}}$. Thus, any $K>2kp^{e_0}$ will be sufficient.
\end{proof}
\begin{lem}
\label{CloseSets}Fix polytopes $Q$, $Q'\subset M_{\mathbb{R}}$ with extremal points in $M'$. For each $e$, set $B_e = (Q_1\cap\frac{1}{p^e}M)+(Q_2\cap\frac{1}{p^e}M)$, $B'_e = (Q_1+Q_2)\cap \frac{1}{p^e}M$. Then there is a constant $K$ such that for each $e$, and for $\vec{u}\in B'_e$, $d(\vec{u}, B_e)< \frac{K}{p^e}$.
\end{lem}
\begin{proof}
It suffices to show that given $\vec{v}\in Q_1+Q_2$, $d(\vec{v}, Q_1\cap\frac{1}{p^e}M+Q_2\cap\frac{1}{p^e}M)$ is bounded above by some $\frac{K}{p^e}$. Fix such $\vec{v}$. Then $\vec{v}=\vec{u}+\vec{w}$ with $\vec{u}\in Q_1$ and $\vec{w}\in Q_2$. By Lemma \ref{PolytopeLatticePoints}, for some constants $K_1$, $K_2$, we have $d(\vec{u},Q_1\cap\frac{1}{p^e}M)<\frac{K_1}{p^e}$, and $d(\vec{w},(t\cdot Q_2)\cap\frac{1}{p^e}M)<\frac{K_2}{p^e}$. It follows that $d(\vec{v}, Q_1\cap\frac{1}{p^e}M+Q_2\cap\frac{1}{p^e}M)<\frac{K_1+K_2}{p^e}$.
\end{proof}
\begin{lem}
\label{HalfSpaceComplement} Let $X$ be a polyhedron and $\vec{v}\notin X$. Then for some half-space $H$, $\vec{v}+H$ does not intersect $X$.
\end{lem}
\begin{proof}
Since $\vec{v}\notin X$, $\vec{v}$ lies in the complement, $H$, of some defining half-space of $X$. Thus, $\vec{v} + H$ does not intersect $X$.
\end{proof}
\begin{lem}
\label{BallInHalfSpace} Fix $K>0$. There is a constant $\kappa$ sufficiently large that for any half-space $H\subset M_{\mathbb{R}}$, $H\cap \mathcal{B}(\vec{0},\kappa)$ contains an open ball of radius $K$ around a lattice point $\vec{w}\in H\cap M$.
\end{lem}
\begin{proof}
Fix $H$. For some vector $\vec{u}$, $H = \{\vec{v}\in M_{\mathbb{R}}\suchthat\vec{v}\cdot \vec{u} \geq 0\}$. We may assume $|\vec{u}| = 1$. Let $\vec{e}_i$ be the standard basis vectors. Then for some $i$, $\frac{1}{n}\leq |\vec{u}\cdot \vec{e}_i|$. If necessary, replace $\vec{e}_i$ with $-\vec{e}_i$ so that $\vec{u}\cdot \vec{e}_i>0$. Set $\kappa = \lceil nK\rceil+K$. Then $\vec{w}=\lceil nK\rceil\vec{e}_i$ satisfies the desired conditions: $\vec{w}\in H\cap \mathcal{B}(\vec{0}, \kappa)$; since $\vec{w}\cdot \vec{u}\geq K$, $\mathcal{B}(\vec{w}, K)\subset H$; and clearly $\mathcal{B}(\vec{w}, K)\subset \mathcal{B}(\vec{0}, \lceil nK\rceil+K)$.
\end{proof}
\begin{lem}
\label{SetDifferences}Suppose we are in the situation of Lemma \ref{CloseSets}. Set $P = Q_1+Q_2$. Then for some constant $\kappa$, $B'_e\backslash B_e$ is contained in an open neighborhood around the boundary $\partial P$ of $P$, of radius $\frac{\kappa}{p^e}$. (That is, if $\vec{v}\in B'_e\backslash B_e$, then $d(\vec{v}, \partial P)<\frac{\kappa}{p^e}$.)
\end{lem}
\begin{proof}
Let $K$ be a constant chosen as in Lemma \ref{CloseSets} so that for all sufficiently large $e$, each monomial in $B'_e$ lies within distance $\frac{K}{p^e}$ of a point in $B_e$. Fix $\kappa$ sufficiently large that for any half-space $H$, $H\cap \mathcal{B}(\vec{0},\kappa)$ contains an open ball of radius $K$ around a lattice point $\vec{w}\in H\cap M$. Suppose for the sake of contradiction that the claim is false. Then there is a monomial $\vec{v}\in B'_e\backslash B_e$ such that $\mathcal{B}(\vec{v}, \frac{\kappa}{p^e})\subset P$. Fix $H$ so that $\vec{v}+H$ does not intersect $P$ (such $H$ exists by Lemma \ref{HalfSpaceComplement}). For our chosen $K$, fix $\vec{w}$ as in Lemma \ref{BallInHalfSpace}.

Let $\vec{u} = \vec{v}+\frac{1}{p^e}\vec{w}$. Since $d(\vec{v}, \vec{u})<\frac{\kappa}{p^e}$, we see that $\vec{u}\in P$, so $\vec{u}\in B'_e$. On the other hand, by our choice of $\kappa$, $\mathcal{B}(\vec{u},\frac{K}{p^e})\subset\vec{v}+H$. It follows (by the statement of Lemma \ref{BallInHalfSpace}) that this ball does not intersect $B_e$, so $d(\vec{u}, B_e)\geq\frac{K}{p^e}$. Thus, by Lemma \ref{CloseSets}, $\vec{u}\notin B'_e$, a contradiction.
\end{proof}

\begin{proof}(Proof of Lemma \ref{TriplesErrorCorrect}.)
We apply Lemma \ref{SetDifferences} with $Q_1 = P_{\sigma}^D$, $Q_2 = t\cdot \Newt(\mathfrak{a})$, so that $\PDa=(Q_1+(-Q_2))\cap \sigma^{\dual}$. By Lemma \ref{SetDifferences}, the difference $a'_e-a_e^{\mathfrak{a}}$ is bounded by the number $n_e$ of $\frac{1}{p^e}M$-lattice points in a neighborhood $\mathcal{B}(\partial \PDa, \frac{\kappa}{p^e})$ of $\partial \PDa$ of radius $\frac{\kappa}{p^e}$. The quantity $\frac{n_e}{p^{ed}}$ is smaller than the volume of the union of all cubes intersecting $\mathcal{B}(\partial \PDa, \frac{\kappa}{p^e})$. That union of cubes is, in turn, contained in $\mathcal{B}(\partial \PDa, \frac{\kappa + n}{p^{ed}})$.

We conclude that $\lim_{e\to\infty}\frac{n_e}{p^{ed}} \leq \lim_{e\to\infty} \Vol(\mathcal{B}(\partial \PDa, \frac{\kappa + \sqrt{n}}{p^{ed}})) = 0$, so  $\lim_{e\to\infty}\frac{n_e}{p^{ed}} = 0$, as we desired to show. 
\end{proof}

\subsection{$\mathbb{Q}$-Gorenstein Triples}
Finally, we prove Corollary \ref{GorensteinTriples}, which gives a particularly nice characterization of $\PDa$ when $(R, D)$ is a $\mathbb{Q}$-Gorenstein.

\begin{cor}\label{GorensteinTriples}
Let $R$ be the coordinate ring of an affine toric variety, $D$ a divisor on $\Spec R$, and $\mathfrak{a}$ a monomial ideal, presented as in Theorem \ref{FSignatureOfTriples}. Suppose that the pair $(R, D)$ is $\mathbb{Q}$-Gorenstein. Then $s(R, D, \mathfrak{a}^t) = \Vol(P_{\sigma}^D\cap t\cdot \Newt(\mathfrak{a}))$.
\end{cor}

\begin{proof}[Proof of Corollary \ref{GorensteinTriples}]
Since the pair $(R, D)$ is $\mathbb{Q}$-Gorenstein, for some $\vec{w}\in M\otimes\mathbb{Q}, \vec{w}\cdot\vec{v}_i=1-a_i$ for each $i$. (Just let $\vec{w}$ be the negative of the vector given by Lemma \ref{QGorensteinPair}.) Set $\phi$ to be the map $\vec{u}\mapsto\vec{w}-\vec{u}$. We claim that $\phi$ is a volume-preserving bijection from $\PDa$ to $(t\cdot\Newt(\mathfrak{a}))\cap P_{\sigma}^D$. The corollary will follow immediately.

Before we prove the claim, we first check that $\phi$ maps $P_{\sigma}^D$ to itself. Suppose $\vec{z}\in P_{\sigma}^D$. Then $0\leq \vec{z}\cdot\vec{v}_i$, so $(\vec{w}-\vec{z})\cdot \vec{v}_i = (1-a_i)-(\vec{z}\cdot\vec{v}_i) \leq 1-a_i$. Similarly, $0\leq (\vec{w}-\vec{z})\cdot \vec{v}_i$. We conclude that $\phi(\vec{z})\in P_{\sigma}^D$.

Returning to our claim: the map $\phi$ is clearly linear, volume-preserving, and self-inverse, so it suffices to show that $\phi(\PDa)= (t\cdot\Newt(\mathfrak{a}))\cap P_{\sigma}^D$. Suppose $\vec{u}\in \PDa$. In particular, $\vec{u}\in P_{\sigma}^D$, so (as we just showed) $\phi(\vec{u})\in P_{\sigma}^D$.

Since $\vec{u}\in \PDa$, we may write $\vec{u} = \vec{x}-\vec{y}$, with $\vec{x}\in P_{\sigma}^D, \vec{y}\in (t\cdot\Newt(\mathfrak{a}))\cap P_{\sigma}^D$. Then $\vec{w}-\vec{u} = \vec{y}+(\vec{w}-\vec{x})$. Since $(\vec{w}-\vec{x})\cdot\vec{v}_i \geq (1-a_i) - (1-a_i) = 0$, we conclude that $\vec{w}-\vec{x}\in \sigma^{\dual}$. Since $t\cdot\Newt(\mathfrak{a})$ is closed under addition by vectors in $\sigma^{\dual}$, we conclude that $\phi(\vec{u})=\vec{y}+(\vec{w}-\vec{x})\in t\cdot\Newt(\mathfrak{a})$. 

So far, we've shown that $\phi(\PDa)\subset P_{\sigma}^D\cap t\cdot\Newt(\mathfrak{a})$. On the other hand, suppose that $\vec{y}\in P_{\sigma}^D\cap t\cdot\Newt(\mathfrak{a})$. We wish to show that $\vec{w}-\vec{y}\in \PDa$. Since $\vec{w}\in P_{\sigma}^D$, $\vec{w}-\vec{y}\in P_{\sigma}^D - ((t\cdot\Newt(\mathfrak{a}))\cap P_{\sigma}^D)$. Moreover, since $\vec{y}\in P_{\sigma}^D$, we have that $\phi(\vec{y})\in P_{\sigma}^D$. We conclude that $\phi(\vec{y})\in (P_{\sigma}^D - ((t\cdot\Newt(\mathfrak{a}))\cap P_{\sigma}^D))\cap P_{\sigma}^D = \PDa$.

It follows that $\phi$ is a volume-preserving bijection. Thus, $s(R, D, \mathfrak{a}^t) = \Vol(\PDa) = \Vol( P_{\sigma}^D\cap t\cdot\Newt(\mathfrak{a}))$, as we desired to show.
\end{proof}


\section{Alternative Monomial Ring Presentations}

\subsection{A Slightly More General $F$-Signature Formula}

Theorem \ref{AffineToricFSignature} can be made to apply to monomial rings that are not quite presented ``torically." In particular, suppose $R = k[S]$, where $S = L\cap \sigma^{\dual}$ for \emph{any} lattice $L$, not just $L = M$. Then it's not difficult to apply a slightly modified version of Theorem \ref{AffineToricFSignature} to this presentation of $R$.

\begin{defn}
Let $\sigma$ be a cone as in Remark \ref{ToricAssumptions}, with primitive generators $\vec{v}_1,\ldots, \vec{v}_r$. Let $L$ be a lattice. For each $i$, let $c_i = \min_{\vec{v}\in L} |\vec{v}\cdot\vec{v}_i|$. We define $P_{\sigma}^L\subset \sigma$ to be the polytope $\{\vec{w}\in M_{\mathbb{R}}\suchthat \forall i, 0\leq \vec{w}\cdot \vec{v}_i< c_i\}$. (Note that if $L = M$, then $P_{\sigma}^L = P_{\sigma}$, as $c_i = 1$ for each $i$.)
\end{defn}

\begin{thm}\label{MoreGeneralFSignature}
(We use the conventions of Remark \ref{ToricAssumptions}.) Let $L\subset M$ be a sublattice, and set $S = \sigma^{\dual}\cap L$. (By Remark \ref{ConeLattice}, $L = \Lattice(S)$.) If $\sigma$ is a full-dimensional cone, then $s(R) = \Vol(P_{\sigma}^L),$ with the volume measured with respect to the lattice $L$. Moreover, for each $e$, $a_e = \#(P_{\sigma}^L\cap \frac{1}{p^e}L)$.
\end{thm}

\begin{proof}
The proof is essentially the same as that of Theorem \ref{AffineToricFSignature} with $M$ replaced by $L$, except that in the supporting Lemma $\ref{PCCharacterization}$, for each $i$, we replace the condition $0\leq \vec{v}_i < 1$ with $0\leq \vec{v}_i < c_i$. (In the original proof, we made use of the fact that $c_i = 1$ for $L = M$. It is easily checked that the proof holds in this more general case if we just replace each 1 with $c_i$ as necessary.)
\end{proof}

\begin{ex}
Theorem \ref{MoreGeneralFSignature} may be used to recover the $F$-signature of a Veronese subring $R^{(n)}$ of a polynomial ring $R = k[x_1,\ldots, x_n]$. (This computation has already been performed in \cite{HL04} and \cite{Sin05}.) In particular, $R^{(n)}=k[\sigma^{\dual}\cap L]$, where $\sigma$ is the first orthant and $L\subset M =\mathbb{Z}^n$ is the lattice of vectors whose coordinates sum to a multiple of $n$. It's easily checked that for such $L$ and $\sigma$, $c_i = 1$ for all $i$, so that $P_{\sigma}^L = P_{\sigma}$. Moreover, $\#(M/L) = n$, so $\Vol_L = \frac{1}{n}\cdot\Vol_M$. We conclude that $s(R^{(n)}) = \Vol_L(P_{\sigma}) = \frac{1}{n}\Vol_M(P_{\sigma}) = \frac{1}{n}s(R)$.
\end{ex}

\subsection{A New Proof of an Old $F$-Signature Formula}

Now we can provide an elementary proof of the $F$-signature formula given by Singh. First, we will need to discuss a few relevant properties of monomial rings.

\begin{defn}
Let $S$ be a semigroup of monomials contained in the semigroup $T$ generated by monomials $x_1,\ldots, x_n$. (So $k[A]$ is the polynomial ring $k[x_1,\ldots, x_n]$.) Then $S$ is \emph{full} if $\Frac{k[S]}\cap k[x_1,\ldots, x_n] = k[S]$. Equivalently, $\Lattice(S)\cap T = S$.
\end{defn}

\begin{defn}
Let $S$ be a semigroup of monomials contained in the semigroup $T$ generated by monomials $x_1,\ldots, x_n$. Then we say that $S$ satisfies property $(*)$ if the following holds: consider any variable $x_i\in T$. Then there exist monomials $\zeta, \eta\in k[S]$ such that $\frac{\zeta}{\eta}$, as a fraction in $\Frac{k[T]}$ in lowest terms, can be written as $\frac{\tau}{x_i}$ (where $\tau$ is a monomial in $S$ but not necessarily in $T$). Equivalently, the lattice $L\subset\mathbb{Z}^n$ generated by $S$ should contain, for each $i$, an element with $i^{th}$ coordinate equal to -1.
\end{defn}

\begin{thm}[Singh]\label{SinghTheorem}
Let $R\subset A = k[x_1,\ldots, x_n]$ be a subring generated by finitely many monomials, $R = k[S]$, where $S$ is a finitely generated semigroup. Suppose $k$ is perfect, and let $m_A$ be the homogeneous maximal ideal of $A$. Assuming that $R$ is presented so that $S$ is full and satisfies property $(*)$, the $F$-signature of $R$ is $s(R) = \lim_{e\to\infty} \frac{l(R/(m_A^{[p^e]}\cap R))}{p^{ed}}.$ In particular, $a_e = l(R/(m_A^{[p^e]}\cap R))$.
\end{thm}

\begin{proof}
We are given that $S = \Lattice(S)\cap T = \Lattice(S) \cap \sigma^{\dual}$, where $\sigma$ is the first orthant, with primitive generators $\vec{v}_i$ equal to the unit vectors in $\mathbb{R}^n$. Thus, we may apply Theorem \ref{MoreGeneralFSignature} to the cone $\sigma$ and the lattice $L = \Lattice(S)$. It remains only to show that $l(R/(m_A^{[p^e]}\cap R)) = \#(P_{\sigma}^L\cap \frac{1}{p^e}L)$. Since $c_i = 1$ for each $i$, the right-hand side is $\#\{\vec{v}\in \frac{1}{p^e}L\suchthat 0\leq \vec{v}\cdot\vec{v}_i<1\}$. The left-hand side is equal to the number of $\vec{v}\in L$ whose coordinates are all less than $p^e$, which is $\#\{\vec{v}\in L\suchthat 0\leq \vec{v}\cdot\vec{v}_i<p^e\}$. Dividing all elements of the left-hand side by $q$, we see that the left-hand side and right-hand side are equal. Thus, $s(R) = \lim_{e\to\infty} \frac{a_e}{p^{ed}} = \lim_{e\to\infty} \frac{l(R/(m_A^{[p^e]}\cap R))}{p^{ed}}.$
\end{proof}

\section{Appendix}

\subsection{Some supporting $F$-signature results}

Remark \ref{ImperfectResidueField} generalizes our main theorem to the case of an imperfect residue field. Theorem \ref{GradedProductFSignature}, computing the $F$-signature of a product, is used in our proof of theorem \ref{AffineToricFSignature} as a means of avoiding discussion of the special case that our toric variety has torus factors. Theorem \ref{LocalProductFSignature} is the local version of Theorem \ref{GradedProductFSignature}. Lemma \ref{GradedToLocal} demonstrates the equivalence of local and $\mathbb{N}$-graded $F$-signature computations.

\begin{rem}\label{ImperfectResidueField}
We wish to extend Theorem \ref{AffineToricFSignature} to the case of an imperfect (but still $F$-finite) residue field. One can show using \cite{Yao06} (Remark 2.3) that $F$-signature is residue field-independent. We give an less general but more concrete argument. Suppose $k$ is not perfect. The arguments of Theorem \ref{AffineToricFSignature} still compute the asymptotic growth rate of the number of splittings of $k[\frac{1}{p^e}S]$: $$\lim_{e\to\infty}\frac{\mbox{free rank}(k[\frac{1}{p^e}S])}{p^{ed}} = \Vol(P_{\sigma}).$$ But for imperfect $k$, $R^{1/p^e} = k^{1/p^e}[\frac{1}{p^e}S]\simeq k^{1/p^e}\otimes_k k[\frac{1}{p^e}S]$. In particular, $R^{1/p^e}$ is a free $k[\frac{1}{p^e}S]$-module of rank $[k^{1/p^e}:k]=p^{e\alpha}$. It follows that the free rank of $R^{1/p^e}$ is $p^{e\alpha}$ times the free rank of $k[\frac{1}{p^e}S]$. Thus, by Definition \ref{FSignatureDefinition}, as well as the above formula, we see immediately that $s(R) = \Vol(P_{\sigma})$.

This approach generalizes to the case of pairs and triples. In the pairs case, let $\Sigma_e^D$ denote the set of monomials in $R(p^eD)^{1/p^e}$. Regardless of whether $k$ is perfect, the arguments of Theorem \ref{FSignatureOfPairs} still compute the asymptotic growth rate of the number of splittings of $k[\Sigma_e^D]$ that also split from $k[\frac{1}{p^e}S]$ to be $\Vol(P_{\sigma}^D)$. However, $R(p^eD)^{1/p^e}\simeq k^{1/p^e}\otimes_k k[\Sigma_e^D]$, so the number of splittings of $R(p^eD)^{1/p^e}$ that also split from $R^{1/p^e}$ is $p^{e(d+\alpha)}\cdot \Vol(P_{\sigma}^D)$, and $s(R, D) = \Vol(P_{\sigma}^D)$, as we desired to show. 

In the triples case, let $\Sigma_e^{D,\mathfrak{a}}$ denote the set of monomials in $((I_e^D)^{1/p^e}:(\overline{\mathfrak{a}^{p^e t}})^{1/p^e})$. Regardless of whether $k$ is perfect, the arguments of Theorem \ref{FSignatureOfTriples} still compute the asymptotic growth rate of $k[\frac{1}{p^e}S]/(\Sigma_e^{D,\mathfrak{a}})$ to be $\Vol(\PDa)$. Now consider Lemma \ref{TriplesInterp}. We see that $F_*^eR/((I_e^D)^{1/p^e}:(\overline{\mathfrak{a}^{p^e t}})^{1/p^e}) \simeq k^{1/p^e}\otimes_k k[\frac{1}{p^e}S]/(\Sigma_e^{D,\mathfrak{a}})$. Thus, $a_e^{\mathfrak{a}} = p^{e(d+\alpha)}\cdot \Vol(\PDa)$, and $s(R, D, \mathfrak{a}) = \Vol(\PDa)$, as we desired to show.
\end{rem}


Now we'll show that the $F$-signature of a product of varieties (i.e., the $F$-signature a tensor product of rings over the appropriate field) is the product of the $F$-signatures. We'll give proofs in both the local and graded cases.

\begin{thm}\label{GradedProductFSignature}
Let $A$ be a semigroup (e.g., $\mathbb{N}$ or $\mathbb{Z}^n$). Let $R$ and $S$ be $A$-graded rings containing a perfect field $k$, each with zeroeth graded piece equal to $k$. Then $s(R\otimes_k S) = s(R)\cdot s(S)$. 
\end{thm}
\begin{proof}
First, note that since $k$ is perfect, $(R\otimes_k S)^{1/p^e} \simeq R^{1/p^e}\otimes_k S^{1/p^e}$. Suppose that $R^{1/p^e}\simeq_{R-mod} R^{\oplus a_e}\oplus M_e$, and $S^{1/p^e}\simeq_{S-mod} S^{\oplus b_e}\oplus N_e$. Then $(R\otimes_k S)^{1/p^e}\simeq_{(R\otimes_k S)-mod} (R^{\oplus a_e}\oplus M_e)\otimes_k (S^{\oplus b_e}\oplus N_e)\simeq (R\otimes_k S)^{\oplus a_e b_e}\oplus (R\otimes_k N_e)^{\oplus a_e}\oplus (M_e\otimes_k S)^{\oplus b_e}$. By Lemma \ref{GradedTensorSummand}, $(R\otimes_k N_e)$ and $(M_e\otimes_k S)$ have no free summands. It follows immediately that the free rank of $(R\otimes_k S)^{1/p^e}$ is $a_e b_e$. Since $\dim (R\otimes_k S) = \dim R+\dim S$, we conclude that $s(R\otimes_k S) = \lim_{e\to\infty}\frac{a_e b_e}{p^{e(\dim R+\dim S)}} = s(R)\cdot s(S)$.
\end{proof}

\begin{lem}\label{GradedTensorSummand}
Let $A$ be a semigroup, and let $R$ and $S$ be $A$-graded rings over a field $k$ with zeroeth graded piece equal to $k$. Let $M$ and $N$ be graded $R$- and $S$-modules, respectively. Suppose $M\otimes_k N$ has a free summand as a graded $R\otimes_k S$-module. Then both $M$ and $N$ have free summands as graded $R$- and $S$-modules.
\end{lem}
\begin{proof}
Suppose that $M\otimes_k N$ has a free summand. Then we have a graded map $\phi: M\otimes_k N\onto R\otimes_k S$. In particular, $1\in \im \phi$. Since $\im\phi$ is generated by the images of graded simple tensors in $M\otimes_k N$, it follows that there is a graded simple tensor $x\otimes y\in M\otimes_k N$ such that $\deg\phi(x\otimes y)=0$ but $\phi(x\otimes y)\neq 0$. The degree zero part of $R\otimes_k S$ is isomorphic to $k$, so (after replacing $\phi$ by $\frac{1}{\phi(x\otimes y)}\phi$) we may assume that $\phi(x\otimes y) = 1\in R\otimes_k S$. 

Now, 1 generates a free summand $R\otimes 1$ of the free $R$-module $R\otimes_k S$, so there is an $R$-module map $\psi: R\otimes_k S\onto R$ sending $\phi(x\otimes y)\mapsto 1$. Consider the map $\psi\circ\phi: M\to R\otimes_k S\to R$. This map sends $x\mapsto \phi(x\otimes y)\mapsto 1$. We conclude that $M$ has a free summand as an $R$-module. By a symmetric argument, $N$ has a free summand also.
\end{proof}

\begin{thm}\label{LocalProductFSignature}
Suppose $(R, m_R, k)$ and $(S, m_S, k)$ are local rings containing the same perfect residue field $k$. Let $m\subset R\otimes_k S$ be the maximal ideal $m_R\otimes_k S + R\otimes_k m_S$. Then $s((R\otimes_k S)_m) = s(R)\cdot s(S)$.
\end{thm}

\begin{proof}
First, note that since $k$ is perfect, $(R\otimes_k S)^{1/p^e} \simeq R^{1/p^e}\otimes_k S^{1/p^e}$. Suppose that $R^{1/p^e}\simeq_{R-mod} R^{\oplus a_e}\oplus M_e$, and $S^{1/p^e}\simeq_{S-mod} S^{\oplus b_e}\oplus N_e$. Then $(R\otimes_k S)_m^{1/p^e}\simeq_{(R\otimes_k S)_m-mod} ((R^{\oplus a_e}\oplus M_e)\otimes_k (S^{\oplus b_e}\oplus N_e))_m\simeq (R\otimes_k S)_m^{\oplus a_e b_e}\oplus (R\otimes_k N_e)_m^{\oplus a_e}\oplus (M_e\otimes_k S)_m^{\oplus b_e}$. By Lemma \ref{TensorSummand}, $(R\otimes_k N_e)_m$ and $(M_e\otimes_k S)_m$ have no free summands. It follows immediately that the free rank of $(R\otimes_k S)_m^{1/p^e}$ is $a_e b_e$. Since $\dim (R\otimes_k S)_m = \dim R+\dim S$, we conclude that $s((R\otimes_k S)_q) = \lim_{e\to\infty}\frac{a_e b_e}{p^{e(\dim R+\dim S)}} = s(R)\cdot s(S)$.
\end{proof}

\begin{lem}\label{TensorSummand}
Let $(R, m_R, k)$ and $(S, m_S, k)$ be local rings with residue field $k$. Let $m\subset R\otimes_k S$ be the maximal ideal $m_R\otimes_k S + R\otimes_k m_S$. Let $M$ and $N$ be $R$- and $S$-modules, respectively. Suppose $(M\otimes_k N)_m$ has a free summand as an $(R\otimes_k S)_m$-module. Then both $M$ and $N$ have free summands as $R$- and $S$-modules.
\end{lem}
\begin{proof}
Suppose that $(M\otimes_k N)_m$ has a free summand. Then we have a map $\phi: (M\otimes_k N)_m\onto (R\otimes_k S)_m$. Since $\phi\in \Hom_{((R\otimes_k S)_m)}((M\otimes_k N)_m, (R\otimes_k S)_m)\simeq (\Hom_{R\otimes_k S}(M\otimes_k N, R\otimes_k S))_m$, by clearing denominators we may assume that $\phi$ maps $M\otimes_k N\to R\otimes_k S$ so that $\im \phi\not\subset m(R\otimes_k S)$. It follows that there is a simple tensor $x\otimes y\in M\otimes_k N$ such that $\phi(x\otimes y)\not\in m(R\otimes_k S)$. In particular, $\phi(x\otimes y)\not\in m_R(R\otimes_k S)$. Recall that if $A$ is a free module over a Noetherian local ring $(R, m_R)$, then each element of $A\backslash m_R A$ generates a free summand of $A$. Thus, $\phi(x\otimes y)$ generates a free summand of the free $R$-module $R\otimes_k S$, and there is an $R$-module map $\psi: R\otimes_k S\onto R$ sending $\phi(x\otimes y)\mapsto 1$.

Now consider the map $\psi\circ \phi: M\to R\otimes_k S\to R$. This map sends $x\mapsto \phi(x\otimes y)\mapsto 1$. We conclude that $M$ has a free summand as an $R$-module. By a symmetric argument, $N$ has a free summand also.
\end{proof}

The following lemma (in particular, Lemma \ref{GradedToLocal}.\ref{GradedEqualsLocal}) demonstrates that if $R$ is $\mathbb{N}$-graded with homogeneous maximal ideal $m$, then the (graded) $F$-signature of $R$ and the (local) $F$-signature of $R_m$ are equal. Moreover, the ideal $I_e\subset R_m$ has a ``graded counterpart" $I_e^{\gr}\subset R$ which may be used to define $F$-signature in the graded category.

\begin{lem}\label{GradedToLocal}
Let $R$ be an $\mathbb{N}$-graded ring over a field $k$ and homogeneous maximal ideal $m$. Set $I_e^{\gr} = \{r\in R\suchthat \forall \phi\in \Hom_R(R^{1/p^e}, R), \phi(r^{1/p^e})\in m\}$. let $I_e\subset R_m$ be the ideal defined in Definition \ref{IeDef}, and $i: R\to R_m$ the natural map. Then:
\begin{enumerate}
\item $F_*^eI_e^{gr}$ is the kernel of the map $\psi: F_*^eR\to\Hom_R(\Hom_R(F_*^eR, R), R/mR)$ given by $x\mapsto (\phi\mapsto \phi(x) + mR)$.
\item $I_e^{\gr}$ is homogeneous.
\item $I_e^{\gr} = i^{-1}I_e = I_e \cap R$.
\item If $a_e(R)$ is defined as in \ref{FSignatureDefinition}, then $a_e(R) = l(F_*^eR/F_*^eI_e^{\gr})$.
\item \label{GradedEqualsLocal} For each $e$, $a_e(R) = a_e(R_m)$. Thus, $s(R) = s(R_m)$.
\end{enumerate}
\end{lem}
\begin{proof}
\begin{enumerate}
\item This follows immediately from the definition.
\item Since $F_*^eR$, $R$, and $R/m$ are graded $R$-modules, so are the modules of homomorphisms between them. The map $\psi$ is degree-preserving, that is, it's homogeneous of degree zero. It follows that $F_*^eI_e^{gr}=\ker\psi$ is a graded submodule. Hence, $I_e^{\gr}$ is a homogeneous ideal.
\item Clearly $F_*^eI_e = \ker(\psi\otimes_R R_m)$. It follows that $F_*^eI_e \simeq F_*^eI_e^{\gr}\otimes_R R_m$ (by the flatness of localization). Thus, $I_e = I_e^{\gr}R_m = i(I_e^{\gr})$. Also, $i:R\to R_m$ is injective. (If $i(x) = 0$, then $wx = 0$ for $w\notin m$. If $x_i$ is the lowest-degree term of $x$, and $w_0$ is the degree-zero term of $w$, then $w_0x_i = 0$, but $w_0\neq 0\in k$, so $x_i = 0$. Thus, $x=0$.) We conclude that $I_e^{\gr} = I_e\cap R$.
\item Suppose $F_*^eR = R^{\oplus a_e}\oplus M_e$, where the decomposition is graded and $M_e$ has no graded free summands. We'll show that $F_*^eI_e^{\gr} = mR^{\oplus a_e}\oplus M_e$, from which the claim follows immediately. Since $I_e^{\gr}$ is graded, to compute $I_e^{gr}$ or $F_*^eI_e^{gr}$, we need only check which homogeneous elements they contain. Suppose $x\in F_*^eI_e^{\gr}$. Then $\phi(x) \notin m$ for some $\phi$. Without loss of generality, we may assume that $\phi$ is homogeneous. Then $\phi(x)$ is homogeneous and not in $m$, hence $\phi(x)\in k$, so $x$ generates a direct summand $Rx$ of $F_*^eR$ that splits off from $F_*^eR$. This cannot occur if $x\in mR^{\oplus a_e}\oplus M_e$, so $mR^{\oplus a_e}\oplus M_e\subset F_*^eI_e^{\gr}$. On the other hand, suppose $x$ is homogeneous and $x\notin mR^{\oplus a_e}\oplus M_e$. This occurs precisely when $x$ is a generator of one of the copies of $R$ in the decomposition. In that case, clearly $x\notin F_*^eI_e^{\gr}$. We conclude that $F_*^eI_e^{gr} = mR^{\oplus a_e}\oplus M_e$. The claim follows immediately.
\item It's easily checked that $m^{[p^e]}R_m\subset I_e$, so $m^{[p^e]}R\subset I_e^{\gr}$. Thus, $\Rad I_e^{[gr]} = m$, so $R/I_e^{\gr} = R_m/I_e^{\gr}R_m$ as $k$-vector spaces. We conclude that $l(R/I_e^{\gr}) = l(R_m/I_e)$. It follows that $a_e(R) = l(F_*R/F_*I_e^{\gr}) = l(F_*R_m/F_*I_e) = a_e(R_m)$, as we desired to show.
\end{enumerate}
\end{proof}

\begin{rem}
Lemma \ref{GradedToLocal} doesn't actually apply to all monomial rings. In particular, a monomial ring arising from a cone $\sigma$ that is not full-dimensional does not admit an $\mathbb{N}$-grading with zeroeth graded piece $k$. On the other hand, we may apply Theorem \ref{GradedProductFSignature} to reduce to the case where Lemma \ref{GradedToLocal} applies. (If $\sigma$ is not full-dimensional, then $k[\sigma^{\dual}\cap M]$ decomposes as a tensor product of the coordinate ring of a torus, which has $F$-signature equal to 1, and a monomial ring arising from a full-dimensional cone, which does admit an $\mathbb{N}$-grading: see Lemma \ref{ConeProduct}.) 
\end{rem}

\bibliography{Frobenius}
\bibliographystyle{amsalpha}

\end{document}